\documentclass[12pt,a4paper]{article}

\usepackage[utf8]{inputenc}
\usepackage{amsmath,amssymb,amsthm,amsfonts,amstext,amsbsy,amsopn,amscd}
\usepackage{amsxtra,latexsym,exscale,verbatim}
\usepackage{paralist}
\usepackage{bm,bbm}
\usepackage{graphicx}
\usepackage{color}
\usepackage{hyperref}

\numberwithin{equation}{section}
\numberwithin{figure}{section}

\newtheorem{theorem}{Theorem}[section]
\newtheorem{lemma}{Lemma}[section]
\newtheorem{defin}{Definition}[section]
\newtheorem{rem}{Remark}[section]

\newtheorem{assump}{Assumption}[section]

\newcommand{\eps}{\varepsilon}
\newcommand{\rr}{\mathbb{R}}
\newcommand{\rs}{\mathbb{S}}
\newcommand{\dd}{\,\mathrm{d}}
\newcommand{\rmL}{\mathrm{L}}
\newcommand{\rmH}{\mathrm{H}}
\newcommand{\rmC}{\mathrm{C}}

\newcommand{\bc}{{\mathbf c}}
\newcommand{\bu}{{\mathbf u}}
\newcommand{\bv}{{\mathbf v}}

\newcommand{\cL}{{\mathcal L}}
\newcommand{\cP}{{\mathcal P}}
\newcommand{\cT}{{\mathcal T}}
\newcommand{\RR}{\mathcal{R}_\eps}
\newcommand{\spec}{\mathrm{spec}}
\newcommand{\Span}{\mathrm{Span}}

\definecolor{ddmagenta}{rgb}{0.7,0,0.9}
\definecolor{ddorange}{rgb}{1,0.5,0}
\definecolor{ALEXBLAU}{rgb}{0.3,0.1,0.9}

\begin{document}

\title{Pulses in FitzHugh--Nagumo systems\\with rapidly oscillating coefficients}
\author{Pavel Gurevich\footnote{Freie Universit\"at Berlin, Institute for Mathematics, Arnimallee 3, 14195 Berlin, Germany}
\footnote{RUDN University, Miklukho-Maklaya 6, 117198, Moscow, Russia}
\quad and \quad
Sina Reichelt\footnote{Weierstrass Institute for Applied Analysis and Stochastics, Mohrenstrasse 39, 10117 Berlin, Germany}}
\maketitle

\begin{abstract}\noindent
This paper is devoted to pulse solutions in FitzHugh--Nagumo systems that are coupled parabolic equations with rapidly periodically oscillating coefficients. In the limit of vanishing periods, there arises a two-scale FitzHugh--Nagumo system, which qualitatively and quantitatively captures the dynamics of the original system. We prove existence and stability of pulses in the limit system and show their proximity on any finite time interval to pulse-like solutions of the original system.
\end{abstract}



\section{Introduction}
\label{sec:intro}

The famous FitzHugh--Nagumo equations, first mentioned in \cite{NAY1962}, model the pulse transmission in animal nerve axons.
The fast, nonlinear elevation of the membrane voltage $u$ is diminished over time by a slower, linear recovery variable $v$. The activator $u$ and the inhibitor $v$ are the solutions of a nonlinear partial differential equation (PDE) coupled with a linear ordinary differential equation (ODE)
\begin{subequations}
\label{eq:system-orig-sub}
\begin{equation}
\label{eq:system-orig}
\tag{\ref{eq:system-orig-sub}.OG}
\begin{aligned}
u_t & = u_{xx} + f(u) - \alpha v , \\
v_t &  =  - bv + \beta u ,
\end{aligned}
\end{equation}
\end{subequations}
where the nonlinearity is typically given by the cubic function $f(u) = u(1 - u)(u - a)$ for $a \in (0, 1)$. The other parameters usually satisfy $\alpha=1$ and $0 < b\leq \beta \ll 1$. The existence of traveling wave solutions, such as pulses and fronts, are well-known for system \eqref{eq:system-orig}, see e.g.\ \cite{McK1970,Hast1976,Carp1977,Hast1982,JKP1991,AK2015} for pulses and \cite{Deng1991,Szmo1991} for fronts.

We are mainly interested in pulse solutions and consider the following FitzHugh--Nagumo system with rapidly oscillating coefficients in space
\begin{subequations}
\label{eq:system-eps-sub}
\begin{equation}
\label{eq:system-eps}
\tag{\ref{eq:system-eps-sub}.S$_\eps$}
\begin{aligned}
u^\eps_t & = u^\eps_{xx} + f(u^\eps) - \alpha(\tfrac{x}{\eps}) v^\eps , \\
v^\eps_t & = \left( \eps^2 d(\tfrac{x}{\eps}) v^\eps_x \right)_x - b(\tfrac{x}{\eps}) v^\eps + \beta(\tfrac{x}{\eps}) u^\eps ,
\end{aligned}
\end{equation}
\end{subequations}
where $x \in \rr$ and $t>0$.
All coefficients belong to the space $\rmL^\infty(\rs)$ with $\rs = \rr/\mathbb{Z}$ being the periodicity cell, which means that they are $1$-periodic on $\rr$. We imagine that these oscillations model heterogeneity within an excitable medium and $\eps > 0$ is the characteristic length scale of the periodic microstructure.
Moreover, in \eqref{eq:system-eps} we allow for a small (slow) diffusion of the inhibitor $v^\eps$, as it is also done in e.g.\ \cite{Szmo1991}. In this paper we study pulse-type solutions in system \eqref{eq:system-eps}, including the case $d \equiv 0$.

To our best knowledge, there are no results in the literature on the existence of \emph{pulses} in FitzHugh--Nagumo systems with periodic coefficients. However, there exists an extensive literature on traveling \emph{fronts} in reaction-diffusion equations with periodic data, see e.g.\ \cite{HuZi1995,BeHa2002} for continuous periodic media, \cite{GuHa2006,CGW2008} for discrete periodic media, and \cite{Xin2000} for a review and further references to earlier works. The article \cite{Hei2001} investigates front solutions  in perforated domains for single equations as well as monotone systems. Most of these results are based on the maximum principle, which fails for the FitzHugh--Nagumo system.
In \cite{MSU2007} reaction-diffusion systems are studied and exponential averaging is used to show that traveling wave solutions can be described by a spatially homogeneous equation and exponentially small remainders.
The existence of generalized (oscillating) traveling waves $u^\eps (t,x) = \bu (x + ct, \frac{x}{\eps})$ and their convergence to a limiting wave $U(t,x) = \bu_0 (x + ct)$ is proved for parabolic equations in \cite{BoMa2014}. In their approach, the authors reformulate the problem as a spatial dynamical system and use a centre manifold reduction. In \emph{all} previous results the limit equation is always ``one-scale''.

Our approach to find pulses in the FitzHugh--Nagumo system \eqref{eq:system-eps} is, first, to  derive an effective system for vanishing $\eps$ and, secondly, to study the existence of pulses in this new system. In the limit $\eps \to 0$, we obtain the following two-scale system
\begin{subequations}
\label{eq:system-lim-sub}
\begin{equation}
\label{eq:system-lim}
\tag{\ref{eq:system-lim-sub}.S$_0$}
\begin{aligned}
U_t & = U_{xx} + f(U) - \int_0^1 \alpha(y) V(t,x,y) \dd y , \\
V_t & = \left( d(y) V_y \right)_y - b(y) V + \beta(y) U ,
\end{aligned}
\end{equation}
\end{subequations}
where $(x,y) \in \rr \times \rs$ and $t>0$.
Notice that $U(t,x)$ only depends on the macroscopic scale $x\in\rr$, whereas $V(t,x,y)$ also depends on the microscopic scale $y \in \rs$.
We prove that this system admits two-scale pulse solutions $(\bu (x + ct), \bv(x + ct,y))$ under certain assumptions on the parameters $(\alpha, \beta, b, d)$. The main idea of the proof is to decompose $\alpha$ into a sum of eigenfunctions of the differential operator $\cL V = \left( d(y) V_y \right)_y - b(y) V$ and to project the $V$-component onto the corresponding eigenspaces. These projections yield a \emph{guiding system}, which is of the form \eqref{eq:system-orig} and is known to possess a stable pulse solution, and a remaining \emph{guided} part, for which we prove the existence and stability of a pulse solution. Moreover, we show that the two-scale pulse $(\bu, \bv)$ is exponentially stable if the pulse of the corresponding guiding system is exponentially stable.
Furthermore, we show that solutions of the original system \eqref{eq:system-eps} satisfy
\begin{align*}
(u^\eps(t,x) , v^\eps(t,x)) = \left( \bu(x + ct), \bv(x + ct, \tfrac{x}{\eps}) \right) + O(\eps)
\quad\text{as } \eps\to0
\end{align*}
for suitable initial conditions and finite times $t \leq T$. These pulse-type solutions have a profile with a periodic microstructure. In other words, the pulse (its inhibitor component) oscillates in time via $\bv (z, \tfrac{z + ct}{\eps})$. Since our approach yields an explicit relation between two-scale pulses and pulses from the guiding system, we are able to provide numerical examples for pulses in both systems, \eqref{eq:system-eps} and \eqref{eq:system-lim}. Interestingly, in one example, a pulse exists, although the microscopic average over $\rs$ of the inhibitor $\bv$ vanishes at every macroscopic point $x\in\rr$.

\emph{This paper is structured as follows.} In Section \ref{sec:justify} we derive the two-scale system and prove $\rmL^2$-error estimates for the difference between the solutions $(u^\eps, v^\eps)$ and $(U,V)$ of \eqref{eq:system-eps} and \eqref{eq:system-lim}, respectively. Section \ref{subsec:pulse-exist} is devoted to the existence of two-scale pulses $(\bu, \bv)$. The stability of these pulses is studied in Section \ref{subsec:stability}. Finally, we provide three numerical examples in Section \ref{sec:numerics}.

\section{Justification of the two-scale system}
\label{sec:justify}

We aim to justify the two-scale FitzHugh--Nagumo system \eqref{eq:system-lim} and derive error estimates for the difference of $(u^\eps,v^\eps)$ and $(U,V)$ being the solutions of the systems \eqref{eq:system-eps} and \eqref{eq:system-lim}, respectively. Since we do not know whether there exist pulses for the original system, arbitrary solutions to coupled parabolic equations are considered in this section. In order to compare the two inhibitors $v^\eps(t,x)$ and $V(t,x,y)$, which depend on different variables, the \emph{macroscopic reconstruction $\RR$} is defined via
\begin{align}
\RR : \rmL^2(\rr; \rmC^0(\rs)) \to \rmL^2(\rr); \quad
(\RR \Phi)(x) := \Phi(x, \tfrac{x}{\eps}) .
\end{align}
We require continuity with respect to at least one of the two variables $(x,y)$ such that the function $\Phi(x,\frac{x}{\eps})$ is measurable on the null-set $\lbrace (x,\frac{x}{\eps}) \,|\, x\in \rr \rbrace \subset \rr\times \rs$. The operator $\RR : \rmC^0(\rr; \rmL^2(\rs)) \to \rmL^2(\rr)$ is also well-defined, see e.g.\ \cite{LNW02} for more details on the regularity of two-scale functions.
To derive quantitative error estimates, we postulate the following assumptions here and throughout the whole text.
\begin{assump}
\label{assump:coeff}
\begin{enumerate}
\item\label{assump:coeff1}
The coefficients satisfy $\alpha, \beta, b \in \rmL^\infty(\rs)$ and $d \in \rmC^1(\rs)$. Moreover, either
\begin{align*}
& \text{(a)} \quad \exists\, d_*>0: \quad d(y) \geq d_*
\quad \text{for all } y \in \rs , \quad\text{or} \\
& \text{(b)} \quad d(y) \equiv 0 .
\end{align*}
\item\label{assump:coeff2} The nonlinear function $f \in \rmC^1(\rr)$ admits the growth conditions
\begin{align*}
f(u) \geq c_1 u - c_2 \quad\text{if } u \leq 0
\quad\text{and}\quad
f(u) \leq c_3 u  + c_4 \quad\text{if } u \geq 0
\end{align*}
for some constants $c_1, c_2,c_3,c_4 \geq 0$.
\end{enumerate}
\end{assump}

A prototype nonlinearity $f : \rr \to \rr$ that we have in mind is
\begin{align}
\label{eq:f}
f(u) = u(1 - u)(u - a)
\quad\text{with}\quad
a \in ( 0,1 ) .
\end{align}
Of course, our theory also applies to other bistable nonlinearities $f$ with similar properties.

Before we derive error estimates, we make sure that unique classical solutions exist. Therefore, the differential operators $\cL_2^\eps: D(\cL_\eps) \to \rmL^2(\rr)$ and $\cL_2^0 : D(\cL_0) \to \rmL^2(\rr\times\rs)$ are introduced via
\begin{align*}
\begin{array}{ll}
(\cL_\eps \varphi)(x) := \left( \eps^2 d(\tfrac{x}{\eps}) \varphi_x \right)_x - b(\tfrac{x}{\eps}) \varphi ,\quad
& D(\cL_\eps) := \lbrace \varphi \in \rmL^2(\rr) \,|\, \cL_\eps \varphi \in \rmL^2(\rr) \rbrace  , \\
(\cL_0 \Phi)(x,y) := \left( d(y) \Phi_y \right)_y - b(y) \Phi
,\quad
& D(\cL_0) := \lbrace \Phi \in \rmL^2(\rr\times\rs) \,|\, \cL_0 \Phi \in \rmL^2(\rr\times\rs) \rbrace .
\end{array}
\end{align*}
Notice that in case (a) of Assumption \ref{assump:coeff}.\ref{assump:coeff1} with microscopic diffusion of the inhibitor, we have $D(\cL_\eps) = \rmH^2(\rr)$ and $D(\cL_0) = \rmL^2(\rr; \rmH^2(\rs))$; in case (b), $D(\cL_\eps) = \rmL^2(\rr)$ and $D(\cL_0) = \rmL^2(\rr\times\rs)$. With a slight abuse of notation, we identify the functions $U(t) \in \rmH^2(\rr)$ and $U(t,x)$, etc.

\begin{defin}
\label{defin:class_sol}
\begin{enumerate}
\item We call $(u^\eps,v^\eps)$ a classical solution of system \eqref{eq:system-eps}, if $(u_\eps,v_\eps)$ is continuous on $[0,T]$, continuously differentiable on $(0,T)$, satisfies $u^\eps(t) \in \rmH^2(\rr)$ and $v^\eps(t) \in D(\cL_\eps)$ for $0<t<T$, and solves on $[0,T]$ the equations
\begin{align*}
\begin{aligned}
& u^\eps_t  = u^\eps_{xx} + f(u^\eps) - \alpha(\tfrac{x}{\eps}) v^\eps ,
& \qquad u^\eps(0) = u^\eps_0 ,\\
& v^\eps_t  = \cL_\eps v^\eps + \beta(\tfrac{x}{\eps}) u^\eps ,
& \qquad v^\eps(0) = v^\eps_0 .
\end{aligned}
\end{align*}
\item We call $(U,V)$ a classical solution of system \eqref{eq:system-lim}, if $(U,V)$ is continuous on $[0,T]$, continuously differentiable on $(0,T)$, satisfies $U(t) \in \rmH^2(\rr)$ and $V(t) \in D(\cL_0)$ for $0<t<T$, and solves on $[0,T]$ the equations
\begin{align*}
\begin{aligned}
& U_t  = U_{xx} + f(U) - \int_0^1 \alpha(y) V(t)(x,y) \dd y ,
& \qquad U(0) = U_0 ,\\
& V_t  = \cL_0 V + \beta(y) U ,
& \qquad V(0) = V_0 .
\end{aligned}
\end{align*}
\end{enumerate}
\end{defin}

We will take initial data for $V$ in the two-scale space
\begin{align*}
\mathbb{V}_{\cL_0} := \lbrace \Phi \in D(\cL_0) \,|\, \Phi_x, \Phi_{xx} \in D(\cL_0) \rbrace \cap \rmL^\infty(\rr\times\rs) .
\end{align*}
Notice that for $d > 0$, there holds $\mathbb{V}_{\cL_0} = \rmH^2(\rr; \rmH^2(\rs))$ and all functions belonging to $\rmH^2(\rr; \rmH^2(\rs))$ are essentially bounded by the Sobolev embeddings $\rmH^2(\rr) \subset \rmL^\infty(\rr)$ and $\rmH^2(\rs) \subset \rmL^\infty(\rs)$. In contrast, for $d = 0$, we need the additional restriction to the set of bounded functions and $\mathbb{V}_{\cL_0} =  \rmH^2(\rr; \rmL^2(\rs)) \cap \rmL^\infty(\rr\times\rs)$.

\begin{assump}
\label{assump:initial-1}
\begin{enumerate}
\item The two-scale initial conditions $(U_0,V_0)$ for system \eqref{eq:system-lim} satisfy $U_0 \in \rmH^2(\rr)$ and $V_0 \in \mathbb{V}_{\cL_0}$.
\item \label{assump:initial-1b}
The one-scale initial conditions $(u^\eps_0 , v^\eps_0)$ for system \eqref{eq:system-eps} satisfy $u^\eps_0 \in \rmH^2(\rr)$ and $v^\eps_0 \in D(\cL_\eps) \cap \rmL^\infty(\rr)$, and fulfill the estimate
\begin{align*}
\exists\,C\geq 0 : \quad
\Vert u^\eps_0 - U_0 \Vert_{\rmL^2(\rr)} + \Vert v^\eps_0 - \RR V_0 \Vert_{\rmL^2(\rr)} \leq \eps C
\quad \text{for all } \eps \in (0,1] .
\end{align*}
\end{enumerate}
\end{assump}
Notice that $V_0 \in \rmC^1(\rr; \rmL^2(\rs))$, thanks to the Sobolev embedding $\rmH^2(\rr) \subset \rmC^1(\rr)$, so that $\RR V_0$ is indeed well defined.
Under the above assumptions, we obtain the existence of classical solutions via the semigroup theory.

\begin{theorem}
\label{thm:sol-exist}
Let Assumptions \ref{assump:coeff} and \ref{assump:initial-1} hold. Then the following is true.

(i) For every $T>0$ and $\eps>0$, there exists a unique classical solution $(u^\eps,v^\eps)$ of system \eqref{eq:system-eps}. Moreover,
\begin{align}
\label{eq:bound-eps}
\Vert (u^\eps, v^\eps) \Vert_{\rmC^1( [0,T]; \rmL^2(\rr))} +
\Vert ( u^\eps_x, \eps v^\eps_x) \Vert_{\rmL^2((0,T)\times\rr))} +
\Vert ( u^\eps, v^\eps) \Vert_{\rmL^\infty((0,T)\times\rr))}
\leq C
\end{align}
for some constant $C = C(T) > 0$ independent of $\eps$.

(ii) For every $T>0$, there exists a unique classical solution $(U,V)$ of the two-scale system \eqref{eq:system-lim}. In addition, the inhibitor satisfies $V \in C^0([0,T]; \mathbb{V}_{\cL_0})$.
\end{theorem}
\begin{proof}
For arbitrary $M >0$, we define the function $f_M : \rr \to \rr$ via
\begin{align*}
f_M(u) := \left\lbrace
\begin{array}{ll}
f(-M) + f'(-M)u \quad & \text{if } u < -M , \\
f(u) & \text{if } |u| \leq M , \\
f(M) + f'(M)u & \text{if } u > M .
\end{array}\right.
\end{align*}
Notice that $f_M \in \rmC^1(\rr)$ is globally Lipschitz continuous.
Then for every $T>0$, the existence of unique classical solutions $(u^\eps_M, v^\eps_M)$ and $(U_M, V_M)$ according to Definition \ref{defin:class_sol} follows from the semigroup theory, see e.g.\ \cite[Sec.~6.1, Thm.~1.5]{Paz83}. The higher regularity $x \mapsto V(t,x,y) \in \rmH^2(\rr)$ follows by taking finite differences as in \cite[Prop.~2.3.17]{Diss-SR}. According to Lemma \ref{lemma:max-bound} and Remark \ref{rem:A1}, the solutions $u^\eps_M, v^\eps_M, U_M$ are bounded in $\rmC^0([0,T];\rmL^\infty(\rr))$ and $V_M$ in $\rmC^0([0,T]; \rmL^\infty(\rr\times\rs))$ uniformly with respect to $\eps$ and $M$. Hence, the result also holds for the unmodified function $f$.

The upper bound for $\Vert (u^\eps, v^\eps) \Vert_{\rmC^1( [0,T]; \rmL^2(\rr))} +
\Vert ( u^\eps_x, \eps v^\eps_x) \Vert_{\rmL^2((0,T)\times\rr))}$ follows from testing the equations with the solution itself and applying Gr\"onwall's Lemma, see e.g.\ \cite[Sec.~2.1.2]{Diss-SR} or \cite[Sec.~4.1]{MRT14}. The upper bound for $\Vert ( u^\eps, v^\eps) \Vert_{\rmL^\infty((0,T)\times\rr))}$ is immediate from Lemma \ref{lemma:max-bound}.
\end{proof}

Finally, we prove error estimates for the difference of the original solution $(u^\eps, v^\eps)$ and the effective solution $(U, \RR V)$, which justifies our investigation of the two-scale system \eqref{eq:system-lim} in the next section.

\begin{theorem}
\label{thm:error-est}
Let Assumptions \ref{assump:coeff} and \ref{assump:initial-1} hold.
Moreover, let $(u^\eps,v^\eps)$ and $(U,V)$ denote classical solutions of the original system \eqref{eq:system-eps} and the two-scale system \eqref{eq:system-lim}, respectively.
Then for every $T>0$, there exists a constant $C>0$ depending on $(U,V)$ but not $\eps$ such that
\begin{align}
\label{eq:est}
&\Vert u^\eps - U \Vert_{\rmC^0([0,T]; \rmL^2(\rr))} +
\Vert v^\eps - \RR V \Vert_{\rmC^0([0,T]; \rmL^2(\rr))} \leq \eps C , \\
\label{eq:est-grad}
& \Vert u^\eps_x - U_x \Vert_{\rmL^2((0,T)\times\rr))} +
\Vert \eps v^\eps_x - \RR V_y \Vert_{\rmL^2((0,T)\times\rr))}
\leq \eps C .
\end{align}
\end{theorem}
\begin{proof}
For brevity, we write the coefficients as $\alpha_\eps(x) := \alpha(\tfrac{x}{\eps})$, etc.
Subtracting the equations for $u^\eps$ and $U$ and respectively $v^\eps$ and $V$ in \eqref{eq:system-eps} and \eqref{eq:system-lim}, testing with $u^\eps - U$, respectively $v^\eps - \RR V$, and integrating over $\rr$ yields for all $t \in [0,T]$
\begin{align}
\label{eq:error-1a}
\int_\rr (u^\eps - U)_t (u^\eps - U) \dd x
& = \int_\rr \bigg\lbrace (u^\eps - U)_{xx} (u^\eps - U) + [f(u^\eps) - f(U)](u^\eps - U) \nonumber \\
& \hspace{4em} - \bigg[ \alpha_\eps v^\eps - \int_0^1 \alpha(y) V(t,x,y)\dd y \bigg] (u^\eps - U) \bigg\rbrace \dd x
\end{align}
as well as
\begin{align}
\label{eq:error-2a}
\int_\rr (v^\eps - \RR V)_t (v^\eps - \RR V) \dd x
& = \int_\rr \bigg\lbrace  \big( ( \eps^2 d_\eps v^\eps_x )_x - \RR [ (d V_y)_y] \big) (v^\eps - \RR V) \nonumber \\
& \hspace{2em} - b_\eps |v^\eps - \RR V |^2
+ \beta_\eps (u^\eps - U) (v^\eps - \RR V )  \bigg\rbrace \dd x .
\end{align}
In case (a) of Assumption \ref{assump:coeff}.\ref{assump:coeff1}, using the relation $\eps (\RR V)_x = \RR (\eps V_x + V_y)$, we obtain
\begin{align*}
& \left( \eps^2 d_\eps (\RR V)_x \right)_x
 = \RR [( d V_y)_y ]
+ \Delta^\eps , \\
& \Delta^\eps  := \eps \RR [(d V_y)_x]
+ \eps \RR [(d V_x)_y]
+ \eps^2 \RR[(d V_x)_x]
\end{align*}
with $d = d(y)$ and by Theorem \ref{thm:sol-exist} (ii), we find the upper bound
\begin{align}
\label{eq:error-Delta}
\Vert \Delta^\eps \Vert_{\rmL^2(\rr)}
\leq \eps C_1(t)
\quad\text{with}\quad
C_1(t) := \Vert d \Vert_{\rmC^1(\rs)} \Vert V(t,\cdot,\cdot)\Vert_{\rmH^2(\rr; \rmL^2(\rs)) \cap \rmH^1(\rr; \rmH^1(\rs))} .
\end{align}
In case (b) of Assumption \ref{assump:coeff}.\ref{assump:coeff1}, we set $C_1(t) : = 0$.

Applying partial integration with the boundary conditions
\begin{align*}
\lim_{x \to\pm\infty} u^\eps_x(t,x) = 0, \quad
\lim_{x \to\pm\infty} v^\eps_x(t,x) = 0, \quad
\lim_{x \to\pm\infty} U_x(t,x) = 0, \quad
\lim_{x\to\pm\infty}  V_x(t,x,y) = 0 ,
\end{align*}
for all $t \in [0,T]$ and almost all $y \in \rs$,
and the chain rule $\frac12\frac{\dd}{\dd t} \Vert u \Vert^2_{\rmL^2(\rr)} = \int_{\rr} \dot{u} u \dd x$, we see that the two equations \eqref{eq:error-1a} and \eqref{eq:error-2a} take the form
\begin{align}
\label{eq:error-1b}
\frac12 \frac{\dd}{\dd t} \Vert u^\eps - U \Vert^2_{\rmL^2(\rr)}
& = \int_\rr \bigg\lbrace - |u^\eps_x - U_x|^2 + [f(u^\eps) - f(U)](u^\eps - U) \nonumber \\
& \hspace{6em}- \bigg[ \alpha_\eps v^\eps - \int_0^1 \alpha(y) V(t,x,y)\dd y \bigg] (u^\eps - U) \bigg\rbrace \dd x
\end{align}
as well as
\begin{align}
\label{eq:error-2b}
\frac12 \frac{\dd}{\dd t} \Vert v^\eps - \RR V \Vert^2_{\rmL^2(\rr)}
& = \int_\rr \bigg\lbrace  - \eps^2 d_\eps |v^\eps_x - (\RR V)_x|^2
- \Delta^\eps (v^\eps - \RR V) \nonumber \\
& \hspace{5em} - b_\eps |v^\eps - \RR V |^2
+ \beta_\eps (u^\eps - U) (v^\eps - \RR V)  \bigg\rbrace \dd x .
\end{align}
Applying H\"older's and Young's inequality gives
\begin{align}
\label{eq:error-Delta2}
\int_\rr \big| \Delta^\eps (v^\eps - \RR V) \big| \dd x
\leq \frac12 \Vert \Delta^\eps \Vert^2_{\rmL^2(\rr)}
+ \frac12 \Vert v^\eps - \RR V \Vert^2_{\rmL^2(\rr)} .
\end{align}
According to Lemma \ref{lemma:eck} we have for the dual norm
\begin{align}
\label{eq:error-meanVal}
\begin{aligned}
&\Vert \RR (\alpha V) - \textstyle\int_0^1 \alpha(y) V(t,x,y) \dd y \Vert_{\rmH^1(\rr)^*}
\leq \eps C_2(t) \\
& \text{with} \quad
C_2(t) := \Vert \alpha V(t, \cdot, \cdot) \Vert_{\rmH^1(\rr; \rmL^2(\rs))} .
\end{aligned}
\end{align}
Using $\RR(\alpha V) = \alpha_\eps \,(\RR V)$ and \eqref{eq:error-meanVal}, we obtain with H\"older's and Young's inequality
\begin{align}
\label{eq:error-meanVal2}
& \left\vert \int_\rr  \left[ \alpha_\eps v^\eps - \RR (\alpha V) + \RR (\alpha V) - \int_0^1 \alpha(y) V(t,x,y)\dd y \right] (u^\eps - U) \dd x \right\vert \nonumber \\
& \leq \frac12  \Vert \alpha \Vert_{\rmL^\infty(\rs)} \left( \Vert v^\eps - \RR V \Vert^2_{\rmL^2(\rr)} + \Vert u^\eps - U \Vert^2_{\rmL^2(\rr)} \right)
+ \eps^2 \frac12 C_2(t) + \frac12 \Vert u^\eps - U \Vert^2_{\rmH^1(\rr)} .
\end{align}
Using the uniform $\rmL^\infty(\rr)$-bound for $u^\eps, U$ and arguing as in the proof of Theorem \ref{thm:sol-exist}, we can consider $f$ to be globally Lipschitz continuous. Adding \eqref{eq:error-1b} and \eqref{eq:error-2b}, recalling that $d(y) \geq d_* > 0$, and using \eqref{eq:error-Delta}, \eqref{eq:error-Delta2}, and \eqref{eq:error-meanVal2}, we arrive at
\begin{align}
\label{eq:error-3}
& \frac12 \frac{\dd}{\dd t} \left\lbrace \Vert u^\eps - U \Vert^2_{\rmL^2(\rr)}
+ \Vert v^\eps - \RR V \Vert^2_{\rmL^2(\rr)} \right\rbrace +
\Vert u^\eps_x - U_x \Vert^2_{\rmL^2(\rr)} +
\eps^2 d_* \Vert v^\eps_x - (\RR V)_x \Vert^2_{\rmL^2(\rr)} \nonumber \\
& \leq L(t) \left\lbrace \Vert u^\eps - U\Vert^2_{\rmL^2(\rr)} +
\Vert v^\eps - \RR V \Vert^2_{\rmL^2(\rr)}
 + \eps^2 \big( (C_1(t))^2 + (C_2(t))^2 \big) \right\rbrace ,
\end{align}
where $L(t)>0$ depends on the Lipschitz properties of $f$, the upper bound of $\Vert u^\eps (t) \Vert_{\rmL^\infty(\rr)} + \Vert U (t) \Vert_{\rmL^\infty(\rr)}$ in Lemma \ref{lemma:max-bound} and Remark \ref{rem:A1}, as well as $\max \lbrace \Vert \alpha \Vert_{\rmL^\infty(\rs)}, \Vert \beta \Vert_{\rmL^\infty(\rs)}, \Vert b \Vert_{\rmL^\infty(\rs)} \rbrace$.
Applying Gr\"onwall's Lemma with Assumption \ref{assump:initial-1}.\ref{assump:initial-1b} for the initial conditions gives for all $t \geq 0$
\begin{align}
\label{eq:error-4}
\max_{0 \leq t\leq T} \left\lbrace
\Vert u^\eps(t) - U(t) \Vert^2_{\rmL^2(\rr)}
+ \Vert v^\eps(t) - \RR V(t) \Vert^2_{\rmL^2(\rr)} \right\rbrace
\leq \eps^2 C_3(t)e^{\int_0^t L(\tau) \dd\tau} ,
\end{align}
where $C_3 (t) >0 $ is bounded on $[0,T]$ and independent of $\eps$.
Hence, estimate \eqref{eq:est} follows by choosing $t = T$ on the right-hand side in \eqref{eq:error-4} and taking the square root.
Moreover, integrating \eqref{eq:error-3} over $[0,T]$ gives with \eqref{eq:est} the gradient estimate \eqref{eq:est-grad}.
\end{proof}

\begin{rem}
Let us introduce the periodic unfolding operator $\cT_\eps : \rmL^2(\rr) \to \rmL^2(\rr\times \rs)$ following \cite{CDG02}
\begin{equation*}
(\cT_\eps v)(x,y) : = v \left( \eps [\tfrac{x}{\eps}] + \eps y \right) ,
\end{equation*}
where $[x] \in \mathbb{Z}$ denotes the integer part of $x\in\rr$. Noting that $(\cT_\eps \RR V)(x,y) = V(\eps [\tfrac{x}{\eps}] + \eps y , y)$ and $x \mapsto V(x,y)$ is Lipschitz continuous, yields the equivalence
\begin{equation*}
\Vert v^\eps - \RR V \Vert_{\rmL^2(\rr)} \leq \eps C
\qquad\Longleftrightarrow\qquad
\Vert \cT_\eps v^\eps - V \Vert_{\rmL^2(\rr\times\rs)} \leq \eps C .
\end{equation*}
In particular, \eqref{eq:est} implies that the inhibitor $v^\eps$ converges to $V$ strongly in the two-scale sense according to the definition of two-scale convergence in \cite{MT07}. In the same manner, \eqref{eq:est-grad} yields the strong two-scale convergence of $\eps v^\eps_x$ to $\nabla_y V$.
\end{rem}

\section{Pulses in the two-scale system}
\label{sec:pulses}

We seek solutions $(U ,V)$ of the two-scale system \eqref{eq:system-lim} that are frame invariant with respect to the co-moving frame $z = x + \bc t$ such that
\begin{align*}
U(t,x) = \bu (x + \bc t)
\quad\text{and}\quad
V(t,x,y) = \bv (x + \bc t,y) ,
\end{align*}
where $\bc \geq 0$ denotes the constant wave speed.
Inserting this ansatz into \eqref{eq:system-lim} yields the nonlocally coupled system of an ODE and a PDE

\begin{subequations}
\label{eq:ComovingSystem-sub}
\begin{equation}
\label{eq:ComovingSystem}
\tag{\ref{eq:ComovingSystem-sub}.Co-S$_0$}
\begin{aligned}
\bc\bu' & =  \bu'' + f(\bu) - \int_0^1 \alpha(y) \bv(\cdot,y) \dd y , \\
\bc\bv' & =  - \cL \bv + \beta(y) \bu ,
\end{aligned}
\end{equation}
\end{subequations}
where $\bu' = \bu_z$.
The differential operator $\cL: D(\cL) \to \rmL^2(\rs)$ is given via
$$
(\cL \varphi)(y):= - (d(y) \varphi_y)_y + b(y)\varphi
\quad\text{and}\quad
D(\cL):=\{\varphi\in \rmL^2(\rs) \,|\, \cL\varphi\in \rmL^2(\rs) \} .
$$
We denote by $\|\cdot\|_{D(\cL)}$ the graph norm and by $\spec(\cL)$ the spectrum of~$\cL$.
The unknowns of our pulse solution in demand are
\begin{equation}
\label{eq:Unknowns}
\bc\geq 0,\quad \bu:\rr\to\rr,\quad \bv:\rr\times\rs \to \rr .
\end{equation}
\begin{defin}
\label{defin:PulseTS}
The triple $(\bc,\bu (x + \bc t), \bv (x + \bc t,y))$ is called a two-scale pulse solution of the two-scale system~\eqref{eq:system-lim} if $\bu\in \rmC^2(\rr)$, $\bv\in \rmC^0 (\rr;D(\cL))\cap \rmC^1(\rr; \rmL^2(\rs))$, equations~\eqref{eq:ComovingSystem} hold, and $(\bu,\bv)$ is a homoclinic orbit of \eqref{eq:ComovingSystem}, i.e.,
\begin{equation}\label{eqHomoclinicH1}
\lim\limits_{z\to\pm\infty} \bu(z)=0,\quad \lim\limits_{z\to\pm\infty}\|\bv(z,\cdot)\|_{D(\cL)}=0.
\end{equation}
\end{defin}
Throughout this section, we assume the following.
\begin{assump}
\label{assump:EllipticOrNot}
There holds $0\notin\spec(\cL)$. If $d(y) \equiv 0$, then $b(y ) \equiv b_0$ for some $b_0 > 0$.
\end{assump}
Assumptions \ref{assump:coeff}.\ref{assump:coeff1} and \ref{assump:EllipticOrNot} together imply that the spectrum of $\cL$ is discrete and we can find a spectral gap around zero.

\subsection{Existence of two-scale pulse solutions}
\label{subsec:pulse-exist}

In this section, we provide sufficient conditions under which pulse solutions exist and are determined by what we will call a {\em guiding system} of finitely many ODEs.
Our main assumptions that allow us to reduce the nonlocally coupled PDE system~\eqref{eq:ComovingSystem} to a system of  ODEs are as follows.
\begin{assump}
\label{assump:AlphaEigenfunction}
The function $\alpha(y)$ is a finite sum of eigenfunctions of the operator $\cL$, i.e., there exist $m \geq 1$, $\widetilde\alpha_i \in D(\cL)$, and $\lambda_i\in\rr$ such that $\widetilde{\alpha}_1, ..., \widetilde{\alpha}_m$ are linearly independent and
\begin{align*}
\alpha (y) = \sum_{i=1}^m \widetilde\alpha_i (y)
\qquad\text{with}\qquad
\cL \widetilde \alpha_i=\lambda_i \widetilde\alpha_i .
\end{align*}
To be definite, we assume that $\lambda_i>0$.
\end{assump}
Notice that the eigenvalues $\lambda_i$ in Assumption \ref{assump:AlphaEigenfunction} are not assumed to be simple.
With this, we introduce the new parameters (for $i = 1,...,m$)
\begin{equation}
\label{eq:Alpha*Beta*}
\alpha_i:=\| \widetilde\alpha_i\|_{\rmL^2(\rs)}
\quad\text{and}\quad
\beta_i:=\dfrac{(\beta, \widetilde\alpha_i)_{\rmL^2(\rs)}}{\| \widetilde\alpha_i\|_{\rmL^2(\rs)}}.
\end{equation}

\begin{assump}
\label{assump:GuidingSystem}
The ODE system
\begin{subequations}
\label{eq:ComovingSystemGuiding-sub}
\begin{equation}
\label{eq:ComovingSystemGuiding}
\tag{\ref{eq:ComovingSystemGuiding-sub}.GS}
\begin{aligned}
c u'&=u'' + f(u) - \sum_{i=1}^m \alpha_i v_i , \\
c v_i'&=-\lambda_i v_i + \beta_i u, \hspace{4em}  i=1,...,m,
\end{aligned}
\end{equation}
\end{subequations}
where $\alpha_i,\beta_i,\lambda_i \in\rr$ are given by~\eqref{eq:Alpha*Beta*} and Assumption~\ref{assump:AlphaEigenfunction}, admits a homoclinic orbit $(c,u,v_1,...,v_m)$ satisfying
\begin{equation}
\label{eq:UnknownsGuiding}
c\ge 0
\quad\text{and}\quad
u,v_i\in \rmC^\infty(\rr).
\end{equation}
Moreover, there exists $\sigma > 0$ such that
\begin{equation}
\label{eq:HomoclinicGuiding}
\lim\limits_{z\to\pm\infty}e^{\sigma |z|} u(z)=0,\quad \lim\limits_{z\to\pm\infty}e^{\sigma |z|} u'(z)=0, \quad \lim\limits_{z\to\pm\infty} v_i(z)=0 .
\end{equation}
\end{assump}

We will refer to system~\eqref{eq:ComovingSystemGuiding} as to the {\em guiding system}.

\begin{rem}
\label{rem:param}
\begin{enumerate}
\item
System \eqref{eq:ComovingSystemGuiding} is known to possess a homoclinic orbit, e.g., for cubic functions $f$ as in \eqref{eq:f} and certain parameter sets $(\alpha_i,\beta_i,\lambda_i)_{i=1}^m$, c.f.\ \cite{Carp1977} on ``pulses in systems with $l$ fast and $m$ slow equations''. Typically, these parameters are within the range
\begin{align*}
\alpha_i = 1, \qquad 0 < \beta_i \ll 1, \qquad 0 \leq \lambda_i \ll 1 ,
\end{align*}
see also the numerical examples in Section \ref{sec:numerics}.
\item \label{rem:sign}
Interestingly, neither $\alpha(y)$, $\beta(y)$, nor their product $\alpha(y) \beta(y)$ need to be sign preserving. Moreover, the case $\int_0^1\alpha(y)\dd y=0$ is {\em not} excluded in general, unless $\int_0^1 \alpha(y)\beta(y) \dd y = 0$. In the latter case, $\beta_i=0$ and the system~\eqref{eq:ComovingSystemGuiding} decouples and has no homoclinics.
The former case is exemplarily treated in Section \ref{subsec:two-alpha}.
\item Let $b(y) \equiv b_0$. Then, the two-scale system~\eqref{eq:ComovingSystem} takes the form
\begin{equation}
\label{eq:system_b0}
\begin{aligned}
\bc \bu'& = \bu'' + f(\bu) - \int_0^1 \alpha(y) \bv(z,y)\dd y,\\
\bc \bv'& = -b_0 \bv + \beta(y) \bu.
\end{aligned}
\end{equation}
Obviously, \emph{any} $\alpha\in \rmL^2(\rs)$ satisfies Assumption~\ref{assump:AlphaEigenfunction} with $\lambda_0=b_0$. This situation is illustrated by a numerical example in Section \ref{subsec:jumps}.
\end{enumerate}
\end{rem}

The main result of this paper is the following theorem.

\begin{theorem}
\label{thm:PDEPulse}
  Let Assumptions~\ref{assump:coeff}, \ref{assump:EllipticOrNot}, \ref{assump:AlphaEigenfunction}, and~\ref{assump:GuidingSystem} hold. Then the two-scale system \eqref{eq:system-lim} admits a pulse solution $(\bc, \bu, \bv)$ such that
the pair $(\bc,\bu) = (c,u)$ is the same as in~\eqref{eq:UnknownsGuiding} and $\bv$ satisfies the estimate
  \begin{equation}
  \label{eq:PDEPulseVEstimate}
    \|\bv(z,\cdot)\|_{D(\cL)}+\|\bv_z(z,\cdot)\|_{\rmL^2(\rs)} \leq C e^{-\gamma|z|}
    \qquad\text{for }  z\in\rr,
  \end{equation}
where $C,\gamma>0$ do not depend on $z\in\rr$. Moreover, if $\beta \in D(\cL)$, then
  \begin{equation}
  \label{eq:PDEPulseVEstimate2}
	\| \cL \bv(z,\cdot) \|_{D(\cL)} +
     \|\bv_z(z,\cdot)\|_{D(\cL)} \leq C e^{-\gamma|z|}
    \qquad\text{for }  z\in\rr .
  \end{equation}
\end{theorem}
\begin{proof}
The proof is based on the spectral decomposition of the space $\rmL^2(\rs)$ to recover the guiding system and semigroup properties to derive the exponential decay in \eqref{eq:PDEPulseVEstimate}--\eqref{eq:PDEPulseVEstimate2}.

\emph{Step 1: spectral decomposition.} Under Assumption~\ref{assump:EllipticOrNot}, $\cL$ is a sectorial self-adjoint operator. Its spectrum is bounded from below and consists of isolated real eigenvalues, which admit possible multiple geometric multiplicity. The corresponding eigenfunctions form a basis for $\rmL^2(\rs)$. We denote by $e^{-\cL t}$, $t\ge 0$, the analytic semigroup in $\rmL^2(\rs)$ generated by $\cL$.

Set $\Sigma_-:=\spec(\cL)\cap\{\lambda<0\}$ and $\Sigma_+:=\spec(\cL)\cap\{\lambda>0\}$. Let $\cP_i$ be the orthogonal projector onto the eigenspace $\Span(\widetilde\alpha_i)$, $i = 1,...,m$, $\cP_-$ onto the eigenspace corresponding to $\Sigma_-$ and $\sum_{i=1}^m \cP_i + \cP_+$ onto the eigenspace corresponding to $\Sigma_+$. Set $Y_i:=\cP_i(\rmL^2(\rs))$,  $Y_-=\cP_-(\rmL^2(\rs))$, and $Y_+ = \cP_+(\rmL^2(\rs))$. The spaces $Y_i$, $Y_+$, and $Y_-$ are pairwise orthogonal and invariant under $\cL$. Moreover, $Y_i$ and $Y_-$ are finite-dimensional. By Assumption~\ref{assump:AlphaEigenfunction}, the restriction of $\cL$ onto $Y_i$ is a multiplication by $\lambda_i$. Let $\cL_\pm$ denote the restrictions of $\cL$ onto $Y_\pm$. Then, we have (cf.~\cite[Sec.~1.5]{Hen81})
$$
\begin{aligned}
&\cL_-:Y_-\to Y_-\ \text{is bounded, } & &\spec(\cL_-)=\Sigma_-,\\
&D(\cL_+)=D(\cL)\cap Y_+, & &\spec(\cL_+) \subseteq \Sigma_+.
\end{aligned}
$$
Notice that eigenvalues $\lambda_i$ may but need not belong to $\spec(\cL_+)$.
Moreover, due to Assumption~\ref{assump:EllipticOrNot}, there exists $\sigma_\pm > 0$ such that $\Sigma_-$ is below $-\sigma_-$ and $\Sigma_+$ is above $\sigma_+$. Therefore, there exists $C_1>0$ such that
\begin{equation}
\label{eq:ExponentialEstimates}
\begin{aligned}
  \|e^{-\cL_- t}\|_{Y_-}& \le C_1 e^{\sigma_- t}, & & t\leq 0,\\
  \|e^{-\cL_+ t}\|_{Y_+}& \le C_1 e^{-\sigma_+ t}, & & t > 0,
\end{aligned}
\end{equation}
as well as
\begin{equation}
\label{eq:ExponentialEstimatesL}
\begin{aligned}
    \|\cL_+ e^{-\cL_+ t}\|_{Y_+}\le C_1 t^{-1} e^{-\sigma_+ t},\quad t> 0.
\end{aligned}
\end{equation}

\emph{Step 2: orthogonal projection.} Further in the proof, we assume that $c > 0$ in Assumption \ref{assump:GuidingSystem}, whereas the modifications for the case $c = 0$ are obvious.
We will show that the pulse solution for the two-scale system~\eqref{eq:ComovingSystem} is given by $(\bc,\bu,\bv)$, where $(\bc,\bu) = (c,u)$ are as in \eqref{eq:UnknownsGuiding}, and the $\bv$-component is represented via
\begin{equation}
\label{eq:VSeries}
\bv(z,\cdot) = \sum_{i=1}^m \bv_i(z) \cdot \dfrac{\widetilde\alpha_i(\cdot)}{\alpha_i}
+ \bv_+(z) + \bv_-(z)
\qquad\text{for } z \in \rr ,
\end{equation}
where $\bv_i(z)\in \rr$ and $\bv_\pm(z) \in Y_\pm$.
Exploiting the orthogonal decomposition and setting $\beta_i:=\cP_i(\beta)$ as well as $\beta_\pm:=\cP_\pm(\beta)$, we obtain that the co-moving two-scale system~\eqref{eq:ComovingSystem} is equivalent to the system
\begin{equation}
\label{eq:PDEPulse1-4}
\begin{aligned}
& \bc \bu' = \bu'' + f(\bu) - \sum_{i=1}^m \alpha_i \bv_i ,\\
& \bc \bv_i' = - \lambda_i \bv_i + \beta_i \bu , \hspace{4em}  i = 1,...,m,\\
& \bc \bv_\pm' = - \cL_\pm \bv_\pm + \beta_\pm \bu.
\end{aligned}
\end{equation}
By Assumption~\ref{assump:GuidingSystem}, the first $1+m$ equations admit a pulse solution with
$(\bc , \bu , \bv_1,..., \bv_m) = (c,u,v_1,...,v_m)$ given by~\eqref{eq:UnknownsGuiding}. Since $\lambda_i > 0$ and the $\bv_i$'s are bounded, we have
\begin{equation}
\label{eq:V0Explicit}
\bv_i(z) = \dfrac{1}{c}\int_{-\infty}^z e^{-\frac{\lambda_i}{c}(z - \xi)}\beta_i \bu(\xi)\dd\xi.
\end{equation}
Moreover, we set
\begin{equation}
\label{eq:VjExplicit}
\begin{aligned}
\bv_+(z) & :=\dfrac{1}{c}\int_{-\infty}^z e^{-\frac{\cL_+}{c}(z - \xi)}\beta_+ \bu(\xi)\dd\xi , \\
\bv_-(z) & := -\dfrac{1}{c}\int_z^{+\infty} e^{-\frac{\cL_-}{c}(z - \xi)}\beta_- \bu(\xi)\dd\xi.
\end{aligned}
\end{equation}
Since $\bu \in \rmC^1(\rr)$, it follows from~\cite[Sec.~4.3, Thm.~ 3.5]{Paz83} that $\bv_\pm\in \rmC^0(\rr;D(L_\pm))\cap \rmC^1(\rr;Y_\pm)$. Hence, $\bv \in \rmC^0(\rr;D(L))\cap \rmC^1(\rr;\rmL^2(\rs))$.

\emph{Step 3: exponential decay.} Let $C_2>0$ be according to \eqref{eq:HomoclinicGuiding} such that
\begin{align*}
| \bu(z)|\le C_2 e^{-\sigma|z|},\quad z\in\rr.
\end{align*}
Then the estimate of~\eqref{eq:V0Explicit} and~\eqref{eq:VjExplicit} with the help of~\eqref{eq:ExponentialEstimates} shows that there exist $C_3>0$ and $0<\gamma<\min(\sigma,\lambda_i/c,\sigma_\pm/c)$ such that
\begin{equation}
\label{eq:VjEstimate}
| \bv_i(z)| \le C_3 \beta_i e^{-\gamma|z|},\quad \| \bv_\pm(z)\|_{Y_\pm}\le C_3 \|\beta_\pm\|_{Y_\pm} e^{-\gamma|z|}, \quad z\in\rr.
\end{equation}
Additionally using~\eqref{eq:PDEPulse1-4} and the boundedness of $\cL_-$, we can find $C_4>0$ such that
\begin{equation}
\label{eq:VjEstimate1}
| \bv_i'(z)|\le C_4\beta_i e^{-\gamma|z|},\quad
\|\cL_- \bv_-(z)\|_{Y_-} + \| \bv_-'(z)\|_{Y_-}\le C_4 \|\beta_-\|_{Y_-} e^{-\gamma|z|},\quad  z\in\rr .
\end{equation}
To control $\cL_+ \bv_+ (z)$, we represent $\bv_+(z)$ as follows:
\begin{align*}
\bv_+(z) & = \dfrac{1}{c}\int_{-\infty}^z e^{-\frac{\cL_+}{c}(z - \xi)}\beta_+ [ \bu(\xi) - \bu(z) ] \dd\xi
+ \dfrac{1}{c}\int_{-\infty}^z e^{-\frac{\cL_+}{c}(z - \xi)}\beta_+ \bu(z) \dd\xi \\
& =: \bar{\bv}_1(z) + \bar{\bv}_2(z) .
\end{align*}
According to \eqref{eq:HomoclinicGuiding}, we have $|\bu'(z)|\le C_2 e^{-\sigma|z|}$, $z\in\rr$, and hence
\begin{equation*}
\begin{array}{llll}
\text{(a)} &
|\bu(z) - \bu(\xi)|  \leq & \hspace{-7pt} C_2 e^{-\sigma|z|}  |z - \xi |
\quad & \text{for } 0 \geq z \geq \xi , \\
\text{(b)} &
| \bu(z) - \bu(\xi)| \leq & \hspace{-7pt} C_2 e^{-\sigma|\xi|}  |z - \xi |
\quad & \text{for } z \geq \xi \geq 0, \\
\text{(c)} &
| \bu(z) - \bu(\xi)| \leq  & \hspace{-7pt} C_2 |z - \xi |
\quad & \text{for all } z, \xi \in \rr .
\end{array}
\end{equation*}
First, let $z\leq 0$ be fixed. Exploiting relation~\eqref{eq:ExponentialEstimatesL} and (a) yields $C_5 > 0$ such that
\begin{align*}
\Vert \cL_+ \bar{\bv}_1(z) \Vert_{Y_+}
& \leq  \dfrac{1}{c}\int_{-\infty}^z \big\Vert \cL_+ e^{-\frac{\cL_+}{c}(z - \xi)} \big\Vert_{Y_+}
\big\Vert \beta_+ \big( \bu(\xi) - \bu(z) \big) \big\Vert_{Y_+} \dd\xi \\
& \leq C_1 C_2 \Vert \beta_+\Vert_{Y_+}
\int_{-\infty}^z \dfrac{1}{z - \xi} e^{-\frac{\sigma_+}{c}(z - \xi)} e^{-\sigma|z|} |z - \xi| \dd\xi  \\
& \leq C_5 \Vert \beta_+\Vert_{Y_+} e^{-\gamma|z|} .
\end{align*}
Secondly, fix $z > 0$. Proceeding as in the previous estimate and using (b)--(c) yields
\begin{align*}
\Vert \cL_+ \bar{\bv}_1(z) \Vert_{Y_+}
& \leq C_1 C_2 \Vert \beta_+\Vert_{Y_+} \left\lbrace
\int_{-\infty}^0 \dfrac{1}{z - \xi} e^{-\frac{\sigma_+}{c}(z - \xi)} |z - \xi| \dd\xi \right. \\
& \hspace{8em} + \left.
\int_{0}^z \dfrac{1}{z - \xi} e^{-\frac{\sigma_+}{c}(z - \xi)} e^{-\sigma|\xi|} |z - \xi| \dd\xi \right\rbrace \\
& \leq C_5 \Vert \beta_+\Vert_{Y_+} e^{-\gamma|z|} .
\end{align*}
Next, we obtain
similarly to \cite[Sec.~1.2, Thm.~2.4(b)]{Paz83}
\begin{align*}
\cL_+ \bar{\bv}_2(z) = - \dfrac{1}{c} \cL_+ \left( \int_0^\infty e^{-\frac{\cL_+}{c} \xi} \beta_+ \bu(z) \dd\xi \right)
=  \beta_+ \bu(z) .
\end{align*}
Hence, $\Vert \cL_+ \bar{\bv}_2(z) \Vert_{Y_+} \leq C_5 \Vert \beta_+ \Vert_{Y_+} e^{-\sigma |z|}$ for all $z\in\rr$. Combining the estimates for $\cL_+ \bar{\bv}_1$ and $\cL_+ \bar{\bv}_2$, and using once more \eqref{eq:PDEPulse1-4} gives
\begin{equation}
\label{eq:VjEstimate2}
\|\cL_+ \bv_+(z)\|_{Y_+}+\| \bv_+'(z)\|_{Y_+}\le C_4 \|\beta_+\|_{Y_+} e^{-\gamma|z|},\quad  z\in\rr .
\end{equation}
Overall, relations~\eqref{eq:VSeries} as well as \eqref{eq:VjEstimate}--\eqref{eq:VjEstimate2} imply  estimate~\eqref{eq:PDEPulseVEstimate}.

If $\beta \in D(\cL)$, then we have due to \eqref{eq:VjExplicit}
\begin{equation*}
\cL_+ \bv_+(z) =\dfrac{1}{c}\int_{-\infty}^z e^{-\frac{\cL_+}{c}(z - \xi)} \cL_+ \beta_+ \bu(\xi)\dd\xi , \quad
\cL_- \bv_-(z) = -\dfrac{1}{c}\int_z^{+\infty} e^{-\frac{\cL_-}{c}(z - \xi)} \cL_- \beta_- \bu(\xi)\dd\xi.
\end{equation*}
Using the relation $\bc \cL_\pm \bv_\pm' = - \cL_\pm (\cL_\pm \bv_\pm) + \cL_\pm \bu$ together with \eqref{eq:VjEstimate1} and \eqref{eq:VjEstimate2}, yields the improved estimate~\eqref{eq:PDEPulseVEstimate2}.
\end{proof}

\begin{rem}
\label{rem:zero_spectrum}
\begin{enumerate}
\item \label{rem:v_vanish}
Let $b(y) \equiv b_0$ and $\int_0^1 \beta (y) \dd y =0$. Then the two-scale inhibitor $\bv$ is macroscopically vanishing, i.e., $\int_0^1 \bv(z,y) \dd y \equiv 0$ for all $z\in \rr$. This is immediate from integrating the $\bv$-equation in \eqref{eq:system_b0} over $\rs$. The example in Section \ref{subsec:two-alpha} illustrates this phenomenon.
\item \label{rem:exact_sol}
Let the parameters $(\alpha, \beta, b, d)$ satisfy
\begin{align*}
|\alpha(y)|^2 \equiv 1 , \quad
\beta(y) = \beta_1 \alpha(y), \quad
b(y) \equiv \lambda_1, \quad
d(y) \equiv 0 .
\end{align*}
Then the original system \eqref{eq:system-eps} admits indeed a \emph{generalized pulse solution} $(u^\eps, v^\eps)$ of the form
\begin{align}
\label{eq:genetal_pulse}
u^\eps(t,x) = u(x + ct)
\qquad\text{and}\qquad
v^\eps(t,x) = \alpha(\tfrac{x}{\eps}) v_1(x + ct) ,
\end{align}
where $u^\eps$ is independent of $\eps$, whenever $(c,u,v_1)$ is a homoclinic orbit for the guiding system
\begin{align}
\label{eq:num_guide}
c u' = u'' + f(u) - v_1 , \qquad
c v_1' = - \lambda_1 v_1 + \beta_1  u .
\end{align}
Section \ref{subsec:jumps} provides one example for such a generalized pulse solution.
\item The case of not exactly periodic coefficients such as $\alpha(x,\frac{x}{\eps})$ with $\alpha \in\rmC^\infty (\rr\times\rs)$ is in principle also manageable with our approach, however, the existence of homoclinic orbits for guiding systems with heterogeneous coefficients is beyond the scope of the present paper.
\item In the case where $\beta$ is orthogonal to all eigenfunctions $\widetilde{\alpha}_i$, $i=1,..,m$, all coefficients $\beta_i$ vanish and the equations for $v_i$ decouple from the activator $u$ in the guiding system \eqref{eq:ComovingSystemGuiding}. Then the remaining $u$-equation is of Nagumo type and it is known to possess heteroclinic orbits corresponding to traveling fronts, which can also be found in the two-scale system.
\item The guiding system may admit homoclinic orbits corresponding to multiple pulse solutions, in the sense of \cite{EFF1982}. Since they all satisfy \eqref{eq:HomoclinicGuiding}, our two-scale system \eqref{eq:system-lim} admits multiple pulse solutions as well.
\end{enumerate}
\end{rem}

\subsection{Stability of two-scale pulse solutions}
\label{subsec:stability}

Let us turn our attention back to the full two-scale system \eqref{eq:system-lim}.
By Theorem~\ref{thm:PDEPulse}, it admits the family of pulse solutions
\begin{align}
\label{eq:fam_orig}
( \bu , \bv )_{z_0\in\rr} := \left\lbrace \big( \bu_{z_0} (x + \bc t) , \bv_{z_0} (x + \bc t,y) \big) \,|\, z_0 \in \rr \right\rbrace ,
\end{align}
where $\bu_{z_0} (z) := \bu (z + z_0)$ denotes the shifted function for any shift $z_0 \in \rr$.
Following \cite{EvaI1972,AK2015}, we define exponential stability with respect to the supremum norm for the $z$-variable. For the microscopic variable $y \in \rs$, we distinguish between \emph{weak} exponential stability in $\rmL^2(\rs)$ and \emph{strong} exponential stability in $D(\cL)$.

\begin{defin}
\label{defin:Stability}
\begin{enumerate}
\item\label{defin:StabilityCond}
Let $(U,V)$ denote a solution of the two-scale system \eqref{eq:system-lim} with initial condition $(U_0, V_0)$ and $\mathbb{X}$ denotes a real-valued Hilbert space. We say that the exponential stability condition holds if there exist constants $K_1, K_2, K_3, \kappa >0$ such that for any
\begin{align*}
0 \leq \delta \leq K_1 , \; z_0 \in \rr: \quad
\Vert U_0 - \bu_{z_0} \Vert_{\rmL^\infty(\rr)} + \Vert V_0 - \bv_{z_0} \Vert_{\rmL^\infty(\rr; \mathbb{X})} \leq \delta ,
\end{align*}
there exists a shift $z_1$ with $|z_0 - z_1| \leq \delta K_2$ such that for all $t\geq 0$
\begin{align*}
\Vert U(t, \cdot) - \bu_{z_1}(\cdot + \bc t) \Vert_{\rmL^\infty(\rr)} +
\Vert V(t, \cdot) - \bv_{z_1}(\cdot + \bc t, \cdot) \Vert_{\rmL^\infty(\rr; \mathbb{X})}
\leq \delta K_3 e^{-\kappa t} .
\end{align*}
\item\label{defin:StabilityWeak}
The family of pulse solutions $(\bu, \bv)_{z_0\in\rr}$ in \eqref{eq:fam_orig} is weakly (strongly) exponentially stable, if the exponential stability condition holds with $\mathbb{X} = \rmL^2(\rs)$ (with $\mathbb{X} = D(\cL)$).
\end{enumerate}
\end{defin}

We emphasize that our solutions are bounded according to Theorem \ref{thm:sol-exist}, which justifies the supremum norm in Definition \ref{defin:Stability}. In the case $d(y) \equiv 0$ (no microscopic diffusion), the notions of weak and strong exponential stability coincide.

Furthermore, notice that
\begin{align}
\label{eq:fam_guide}
(u, v_1, ..., v_m)_{z_0\in\rr} := \left\lbrace \big( u_{z_0} (x + ct), v_{1,z_0} (x + ct) , ..., v_{m,z_0}(x + ct) \big) \, |\,  z_0 \in \rr \right\rbrace
\end{align}
with $u$ and $(v_1, ..., v_m)$ given by Assumption \ref{assump:GuidingSystem} is a family of pulse solutions for the standard reaction-diffusion FitzHugh--Nagumo-type system
\begin{subequations}
\label{eq:ComovingSystemGuidingPDE-sub}
\begin{equation}
\label{eq:ComovingSystemGuidingPDE}
\tag{\ref{eq:ComovingSystemGuidingPDE-sub}.GS-PDE}
\begin{aligned}
U_t(t,x) & = U_{xx}(t,x)  + f(U) - \sum_{i=1}^m \alpha_i V_i(t,x) , \\
(V_i)_t(t,x) & =  - \lambda_i V_i(t,x) + \beta_i U(t,x), \quad i=1,...,m.
\end{aligned}
\end{equation}
\end{subequations}
We will refer to system~\eqref{eq:ComovingSystemGuidingPDE} as to the {\em guiding PDE system}.

\begin{assump}
\label{assump:GuidingSystemPDE}
Let $(u, v_1, ..., v_m)_{z_0\in\rr}$ be an exponentially stable family of pulse solutions for the guiding PDE system~\eqref{eq:ComovingSystemGuidingPDE}, i.e., the exponential stability condition in Definition \ref{defin:Stability}.\ref{defin:StabilityCond} holds with $\mathbb{X} = \rr^m$.
\end{assump}

For $m=1$, it is well-known that the pulses of system \eqref{eq:ComovingSystemGuidingPDE} are stable, see e.g.\ \cite{Jon1984} for asymptotic stability and \cite{Yana1985, AK2015} for exponential stability. We expect a similar result to hold true in the case of $m > 1$, however, this is beyond the scope of the present paper.

\begin{theorem}
\label{thm:stability}
Let Assumptions~\ref{assump:coeff}, \ref{assump:EllipticOrNot}, \ref{assump:AlphaEigenfunction}, \ref{assump:GuidingSystem}, \ref{assump:GuidingSystemPDE} hold, and let $\spec(\cL)\subset\{\lambda > 0 \}$. Then the family of pulse solutions $(\bu,\bv)_{z_0\in\rr}$ in \eqref{eq:fam_orig} for the two-scale system~\eqref{eq:system-lim} is weakly exponentially stable.
If $\beta \in D(\cL)$, then $(\bu,\bv)_{z_0\in\rr}$ is also strongly exponentially stable.
\end{theorem}
\begin{proof}
\emph{Step 1: reduction to guiding system.}
Since $\spec(\cL) \subset \{\lambda > 0 \}$, it follows that $\cP_-=0$ and $\cL_-=0$. Therefore, $\bv (z,y)$ is given via the sum in \eqref{eq:VSeries}, where $\bv_-(z,y) \equiv 0$, and $(\bu, \bv_1, ..., \bv_m)_{z_0\in\rr}$ is identical to the family of pulse solutions $(u,v_1, ..., v_m)_{z_0\in\rr}$ in \eqref{eq:fam_guide} for the guiding PDE system \eqref{eq:ComovingSystemGuidingPDE}.
Given the initial conditions $U(0,x) = U_0(x)$ and $V(0,x,y) = V_0(x,y)$, we can decompose the two-scale system~\eqref{eq:system-lim}. Again, the $V$-component is given via the sum
\begin{equation}
\label{eq:VSeriesTime}
V(t,x,\cdot) = \sum_{i=1}^m V_i(t,x) \cdot \frac{\widetilde{\alpha}_i(\cdot)}{\alpha_i}
+ V_+(t,z,\cdot),
\end{equation}
where $V_i(t,x)\in \rr$ and $V_+(t,x):=\cP_+ V(t,x) \in Y_+$. With this, the full two-scale system \eqref{eq:system-lim} reduces to the guiding part
\begin{equation}
\label{eq:ComovingSystemGuidingPDEwithInitialData}
\begin{aligned}
& U_t(t,x) = U_{xx}(t,x)  + f(U) - \sum_{i=1}^m \alpha_i V_i(t,x),\\
& (V_i)_t(t,x) = -\lambda_i V_i(t,x) + \beta_i U (t,x), \qquad  i=1,...,m, \\
& U|_{t=0} = U_0(x),\quad
V_i|_{t=0} = V_{0,i} (x):= \cP_i V_0 (x,\cdot) = \frac{( V_0 (x,\cdot), \widetilde\alpha_i)_{L_2}}{\alpha_i} ,
\end{aligned}
\end{equation}
and the guided part
\begin{equation}
\label{eq:ComovingSystemGuidingPDEwithInitialDataPlus}
\begin{aligned}
& (V_+)_t(t,x) =  -\cL_+ V_+(t,x) + \beta_+ U(t,x),\\
& V_+|_{t=0} = \cP_+ V_0 (x,\cdot).
\end{aligned}
\end{equation}
By Assumption~\ref{assump:GuidingSystemPDE}, there exist constants $K_1, K_2, K_3, \kappa > 0$ such that for any
\begin{align*}
0 \leq \delta \leq K_1, \; z_0 \in \rr : \quad
\Vert U_0  - \bu_{z_0} \Vert_{\rmL^\infty(\rr)} + \sum_{i=1}^m \Vert V_{0,i} -\bv_{i,z_0} \Vert_{\rmL^\infty(\rr)} \leq \delta ,
\end{align*}
there exists a shift $z_1$ with $|z_0 - z_1| \leq \delta K_2$ such that for all $t\geq 0$
\begin{align}
\label{eq:stability}
\Vert U(t) - \bu_{z_1} \Vert_{\rmL^\infty(\rr)} +
\sum_{i=1}^m \Vert V_i(t) - \bv_{i,z_1} \Vert_{\rmL^\infty(\rr)}
\leq \delta K_3 e^{-\kappa t} .
\end{align}
It remains to prove that
\begin{align}
\label{eq:stability2}
\Vert \cP_+ (V_0 - \bv_{z_0}) \Vert_{\rmL^\infty(\rr; \mathbb{X})} \leq \delta
\end{align}
implies for some $K_*, \kappa_* > 0$ and all $t \geq 0$
\begin{align}
\label{eq:Implication}
\Vert V_+(t) - \cP_+ \bv_{z_1} \Vert_{\rmL^\infty(\rr; \mathbb{X})} \leq \delta K_* e^{-\kappa_* t} ,
\end{align}
where $\mathbb{X} = \rmL^2(\rs)$ (and if $\beta\in D(\cL)$, then $\mathbb{X} = D(\cL)$) according to Definition \ref{defin:Stability}.

\emph{Step 2: exponential decay of guided part.} System~\eqref{eq:ComovingSystemGuidingPDEwithInitialDataPlus} is linear and 
$V_+$ is given via
\begin{align}
\label{eq:formula_Vplus}
V_+(t,x) = e^{-\cL_+ t} \big( \cP_+ V_0(x) \big) + \int_0^t e^{-\cL_+ (t - s)} \beta_+ U(s, x) \dd s .
\end{align}
Notice that $\cP_+ \bv = \bv_+$ with $\bv_+$ given in \eqref{eq:VjExplicit}. Since $\bv_+$ solves the $\bv_+$-equation in \eqref{eq:PDEPulse1-4}, we have for all $t \geq 0$ the identity
\begin{align}
\label{eq:identity_Vplus}
\bv_+(x + ct + z_1) = e^{-\cL_+ t} \big( \cP_+ \bv_{z_1} (x) \big) + \int_0^t e^{-\cL_+ (t - s)} \beta_+ \bu_{z_1}(x + cs) \dd s .
\end{align}
Subtracting the equations in \eqref{eq:formula_Vplus} and \eqref{eq:identity_Vplus} as well as using \eqref{eq:ExponentialEstimates} yields
\begin{align}
& \sup_{x\in\rr} \left\Vert V_+(t,x) - \bv_+ (x + ct + z_1) \right\Vert_{Y_+} \nonumber \\
& = \sup_{x\in\rr} \Big\Vert e^{-\cL_+ t} \big( \cP_+ [ V_0(x) - \bv_{z_1} (x) ] \big)
+ \int_0^t e^{-\cL_+ (t-s)} \beta_+ \big[ U (s,x) - \bu_{z_1} (x + cs) \big] \dd s \Big\Vert_{Y_+} \nonumber \\
& \leq C_1 e^{-\sigma_+ t}  \left\lbrace \sup_{x\in\rr} \Vert \cP_+ [ V_0(x) - \bv_{z_0} (x) ] \Vert_{Y_+}
+ \sup_{x\in\rr} \Vert \cP_+ [ \bv_{z_0} (x) - \bv_{z_1} (x) ] \Vert _{Y_+} \right\rbrace
\label{eq:estimate_1} \\
& \quad + C_1 \Vert \beta_+ \Vert_{Y_+} \int_0^t e^{-\sigma_+(t-s)} \sup_{x \in \rr}
\vert U(s, x) - \bu_{z_1}(x + cs) \vert \dd s .
\label{eq:estimate_2}
\end{align}
We estimate the first term in \eqref{eq:estimate_1} by \eqref{eq:stability2} and \eqref{eq:estimate_2} by \eqref{eq:stability}. For the second term in \eqref{eq:estimate_1}, we exploit the Lipschitz continuity
$\Vert \bv_+(z_0) - \bv_+(z_1) \Vert_{Y_+} \leq L |z_0 - z_1| \leq \delta L K_2$ for $\bv_+ \in \rmC^1(\rr; Y_+)$. The Lipschitz constant  $L := \sup_{z\in\rr} \Vert (\bv_+)_z (z, \cdot) \Vert_{Y_+}$ is bounded according to estimate \eqref{eq:PDEPulseVEstimate}.
Choosing $\kappa_* = \min \lbrace \sigma_+, \kappa \rbrace $, we arrive at
\begin{align}
\label{eq:estimate_3}
\sup_{x\in\rr} \left\Vert V_+(t,x) - \bv_+ (x + ct + z_1) \right\Vert_{Y_+}
\leq \delta C_1 \big( 1 + L K_2 + K_3 \Vert \beta_+ \Vert_{Y_+} \big) e^{-\kappa_* t} .
\end{align}
Hence, estimate \eqref{eq:Implication} follows immediately and the weak exponential stability of the family of pulse solutions $(\bu,\bv)_{z_0\in\rr}$ in \eqref{eq:fam_orig} is proven.

If $\beta \in D(\cL)$, then $\bv_+$ belongs according to \eqref{eq:PDEPulseVEstimate2} to the space $\rmC^1(\rr; D(\cL_+))$. With this higher regularity, the estimates \eqref{eq:estimate_1}, \eqref{eq:estimate_2}, and \eqref{eq:estimate_3} also hold with $D(\cL_+)$ instead of $Y_+$. Hence, the family of pulse solutions $(\bu,\bv)_{z_0\in\rr}$ in \eqref{eq:fam_orig} is also strongly exponentially stable.
\end{proof}

\begin{rem}
We point out that the constants $K_3$ and $\kappa$ in Definition \ref{defin:Stability} are in general not the same for the guiding pulse $(u,v_1,..,v_m)_{z_0\in\rr}$ and the two-scale pulse $(\bu, \bv)_{z_0\in\rr}$.
\end{rem}

\section{Numerical simulations}
\label{sec:numerics}

We provide numerical examples for three different parameter settings $(\alpha, \beta, b,d)$ and compare the solutions of the original system \eqref{eq:system-eps} with those of the two-scale system \eqref{eq:system-lim}. In the first two examples the spectrum of $\cL$ is discrete and we know that stable two-scale pulses exist according to Section \ref{sec:pulses}. In the third example $\cL$ has only a continuous spectrum and our \emph{guiding system} approach fails, because the two-scale system does not reduce to finitely many ODEs. However, we observe stable pulse solutions in our simulations.

We numerically solve the FitzHugh--Nagumo equations on the bounded interval $x \in [-300, 300]$ with periodic boundary conditions. We emphasize at this point that the $\eps$ that is chosen in the numerical simulations is in the range $\eps \in [2, 30]$. At first glance, this is not a ``small'' number, however, recall that the characteristic length scale of the microstructure $\eps_\mathrm{char}$ is given by the quotient of microscopic length scale divided by macroscopic length scale. The role of the macroscopic length scale of our system is played by the width of the activator spike, which is about $60$, cf.\ Figure \ref{fig:guide}. With this, the characteristic ratio $\eps_\mathrm{char} \in [0.03,0.5]$ is indeed small.

To calculate the solutions, we implement a semi-implicit discretization scheme in MATLAB. Therein, the diffusion parts are solved via fast Fourier transform and the reaction terms are treated with the explicit Euler method. Therefore, we use the time step $\dd t  = 0.01$. For the spatial 	
discretization we use for the $\eps$-system \eqref{eq:system-eps} the step size $\dd x \approx 0.0366$, and for the two-scale system \eqref{eq:system-lim} $\dd x \approx 1.1742$ and $\dd y \approx 0.0020$.

\subsection{Macroscopically vanishing inhibitor $\bv$}
\label{subsec:two-alpha}

We consider the case of a differential operator $\cL$ with constant coefficients $b(y) \equiv d(y) \equiv \delta$ for $0 < \delta \ll 1$, i.e.,
\begin{align*}
(\cL \varphi) (y) = -\delta (\varphi_{yy} - \varphi)
\quad\text{and}\quad D(\cL) = \rmH^2(\rs) .
\end{align*}
The eigenfunctions of $\cL$ are given via $\varphi^s_n(y) = \sin(2\pi n y)$, $\varphi^c_n(y) = \cos(2\pi n y)$, for $n \geq 1$, and $\varphi_0 (y) \equiv 1$. Therefore, Assumption \ref{assump:EllipticOrNot} is satisfied. The corresponding eigenvalues $\lambda_n = \delta(1 + (2\pi n)^2)$ are isolated, real, positive, and have double geometric multiplicity for all $n \geq 1$, whereas $\lambda_0 = \delta$ is simple.
In this example, $\alpha$ is the sum of two eigenfunctions, namely,
\begin{align}
\label{eq:param}
\begin{array}{l}
\alpha = \widetilde\alpha_1 +\widetilde\alpha_2
\quad\text{with}\quad
\widetilde\alpha_1(y) = \sqrt{2} \sin(2\pi y) ,\quad
\widetilde\alpha_2(y) = \sqrt{2} \sin(4\pi y) , \\
\beta(y) = 0.001 (\alpha(y) + \varphi(y)) , \quad
\varphi(y) = \sqrt{2} \sin(8\pi y) ,
\quad\text{and}\quad
\delta = 0.0001 .
\end{array}
\end{align}
Notice that $\varphi$ is orthogonal to $\alpha$ in $\rmL^2(\rs)$. We emphasize that $\beta$ is not orthogonal to $\alpha$ but the signs of $\alpha$, $\beta$, and the product $\alpha(y) \beta(y)$ are not constant, cf.\ Remark \ref{rem:param}.\ref{rem:sign}.
\begin{figure}[b]
\centering
\includegraphics[width=0.5\textwidth]{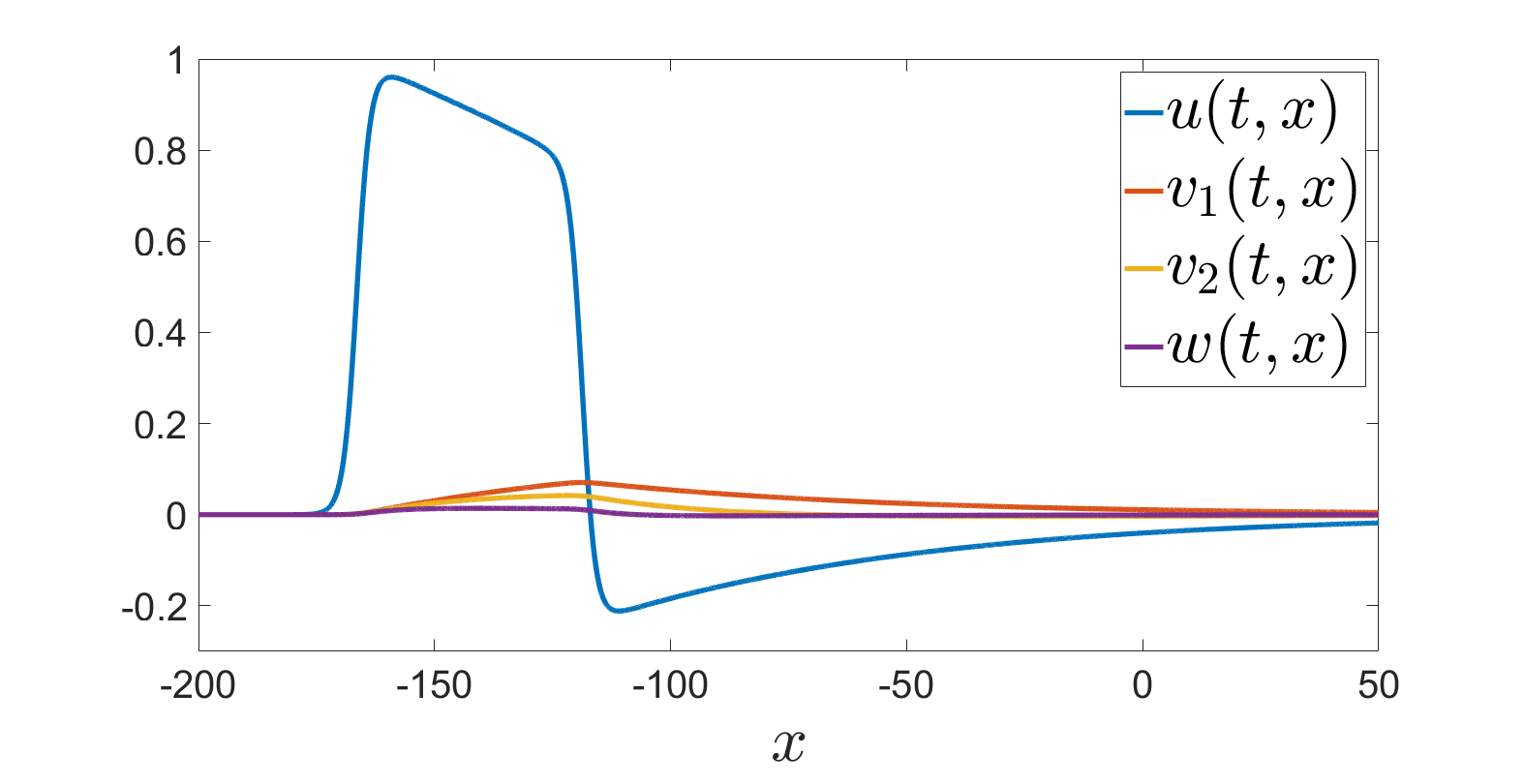}
\caption{Solution $(u,v_1,v_2,w)$ of the guiding system \eqref{eq:guide_time}--\eqref{eq:guide_time_add}. Here and in what follows, the pulse always propagates from the right to the left.}
\label{fig:guide}
\end{figure}

For the choice of parameters in \eqref{eq:param}, the fully decomposed two-scale system of finitely many coupled ODEs as in \eqref{eq:PDEPulse1-4} reads
\begin{align}
\label{eq:guide_time}
& \left\lbrace
\begin{array}{ll}
u_t & \hspace{-5pt} = u_{xx} + u(1 {-} u)(u {-} 0.15) - v_1 - v_2 , \\
(v_1)_t & \hspace{-5pt} = - \lambda_1 v_1 + 0.001 \!\cdot\! u , \\
(v_2)_t & \hspace{-5pt} = - \lambda_2 v_2 + 0.001 \!\cdot\! u ,
\end{array}
\right. \\
\label{eq:guide_time_add}
& \left\lbrace
\begin{array}{ll}
w_t & \hspace{-5pt} = - \lambda_3 w + \beta_+  u , \\
(v_+)_t & \hspace{-5pt} = - \cL_+ v_+ .
\end{array}
\right.
\end{align}
The three-component system \eqref{eq:guide_time} is the guiding system, the $w$-equation in \eqref{eq:guide_time_add} corresponds to the projection onto the eigenfunction $\varphi$, and the $v_+$-equation captures the remaining projections onto the complement of $\Span (\widetilde{\alpha}_1, \widetilde{\alpha}_2, \varphi)$.
In view of \eqref{eq:Alpha*Beta*}, the parameters in the guiding system \eqref{eq:ComovingSystemGuiding} satisfy $\alpha_1 = \alpha_2 = 1$ and $\beta_1 = \beta_2 = \beta_+ = 0.001$. Recall that $\lambda_1 = 0.0001  (1 + 4\pi^2)$, $\lambda_2 = 0.0001 (1 + 16\pi^2)$, and $\lambda_3 = 0.0001  (1 + 64\pi^2)$.

First, we solve the guiding system \eqref{eq:guide_time}--\eqref{eq:guide_time_add}, see Figure \ref{fig:guide},
so that we can use the pulse $(u,v_1,v_2)$ and the additional decoupled component $w$ to compute the initial conditions for the original system \eqref{eq:system-eps} and the two-scale system \eqref{eq:system-lim}.

Secondly, we solve the original system \eqref{eq:system-eps} for various $\eps>0$, see Figure \ref{fig:eps_1},
\begin{equation}
\label{eq:num_eps}
u^\eps_t = u^\eps_{xx} + u^\eps(1 - u^\eps)(u^\eps - 0.15) - \alpha(\tfrac{x}{\eps}) v^\eps , \qquad
v^\eps_t = \delta \left( \eps^2 v^\eps_{xx} - v^\eps \right) + \beta (\tfrac{x}{\eps}) u^\eps ,
\end{equation}
supplemented with the initial condition $u^\eps_0(x) = u(x)$ and $v^\eps_0(x) = \widetilde\alpha_1(\frac{x}{\eps}) v_1(x) + \widetilde\alpha_2(\frac{x}{\eps}) v_2(x) + \varphi(\tfrac{x}{\eps}) w(x)$. According to the homogenization results in Section \ref{sec:justify}, the solutions behave asymptotically like $u^\eps(t,x) = U(t,x) + O(\eps)$ and $v^\eps(t,x) = V(t,x,\frac{x}{\eps}) + O(\eps)$. One can observe in Figure \ref{fig:eps_1} that the amplitude of the oscillations of $u^\eps$ decrease as $\eps$ decreases. However, the amplitude of oscillations of $v^\eps$ does not vanish, while,
smaller $\eps$ lead to higher frequencies. In Figure \ref{fig:eps_1}, we also observe oscillations of the inhibitor $v^\eps$, which correspond to the different modes $\widetilde{\alpha}_1$, $\widetilde{\alpha}_2$, and $\varphi$.
\begin{figure}[h!]
\includegraphics[width=0.5\textwidth]{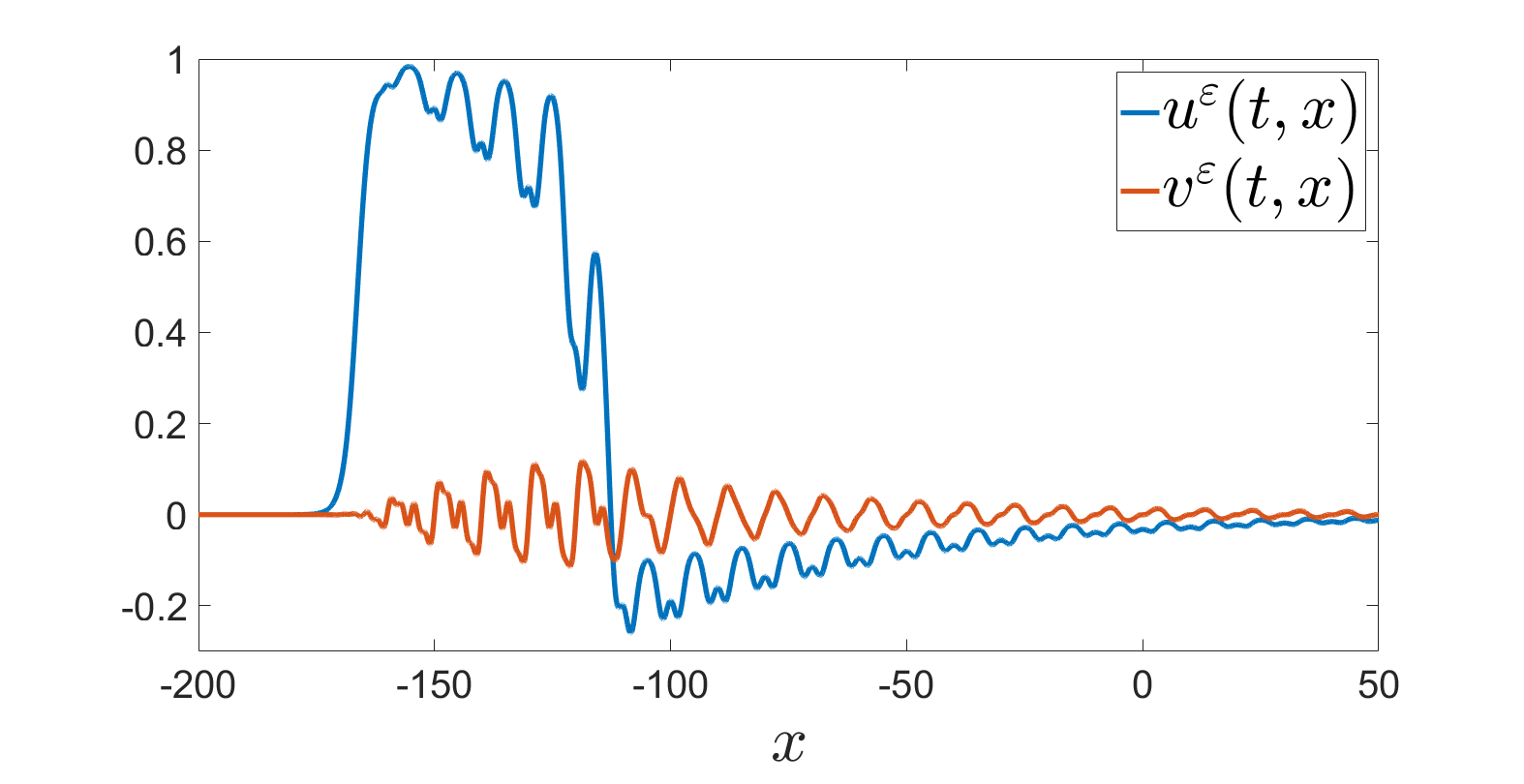}
\includegraphics[width=0.5\textwidth]{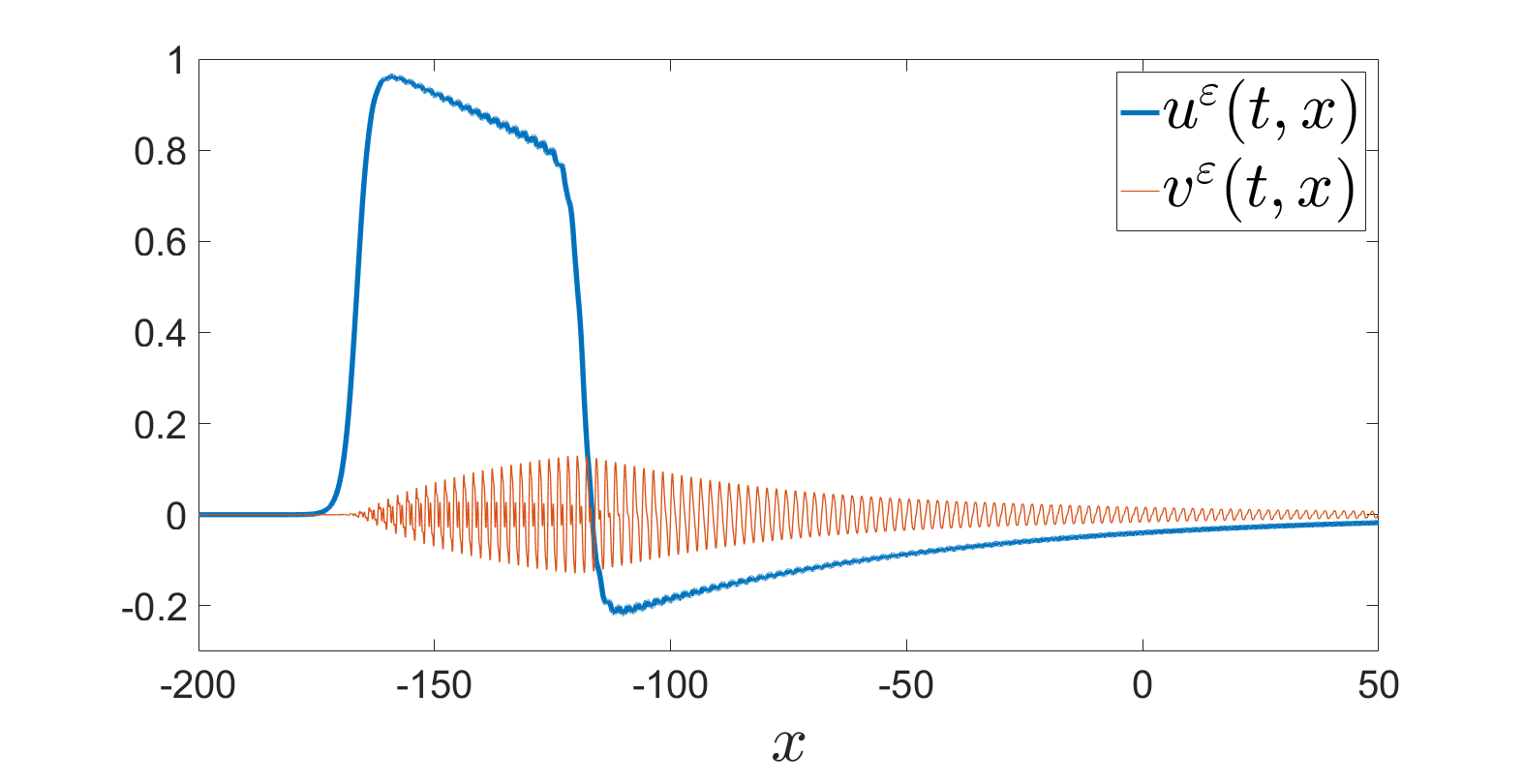}
\caption{Solution $(u^\eps, v^\eps)$ of the original system \eqref{eq:num_eps} with parameters \eqref{eq:param}. Left: $\eps = 10$. Right: $\eps = 2$.}
\label{fig:eps_1}
\end{figure}
\begin{figure}[h!]
\includegraphics[width=0.5\textwidth]{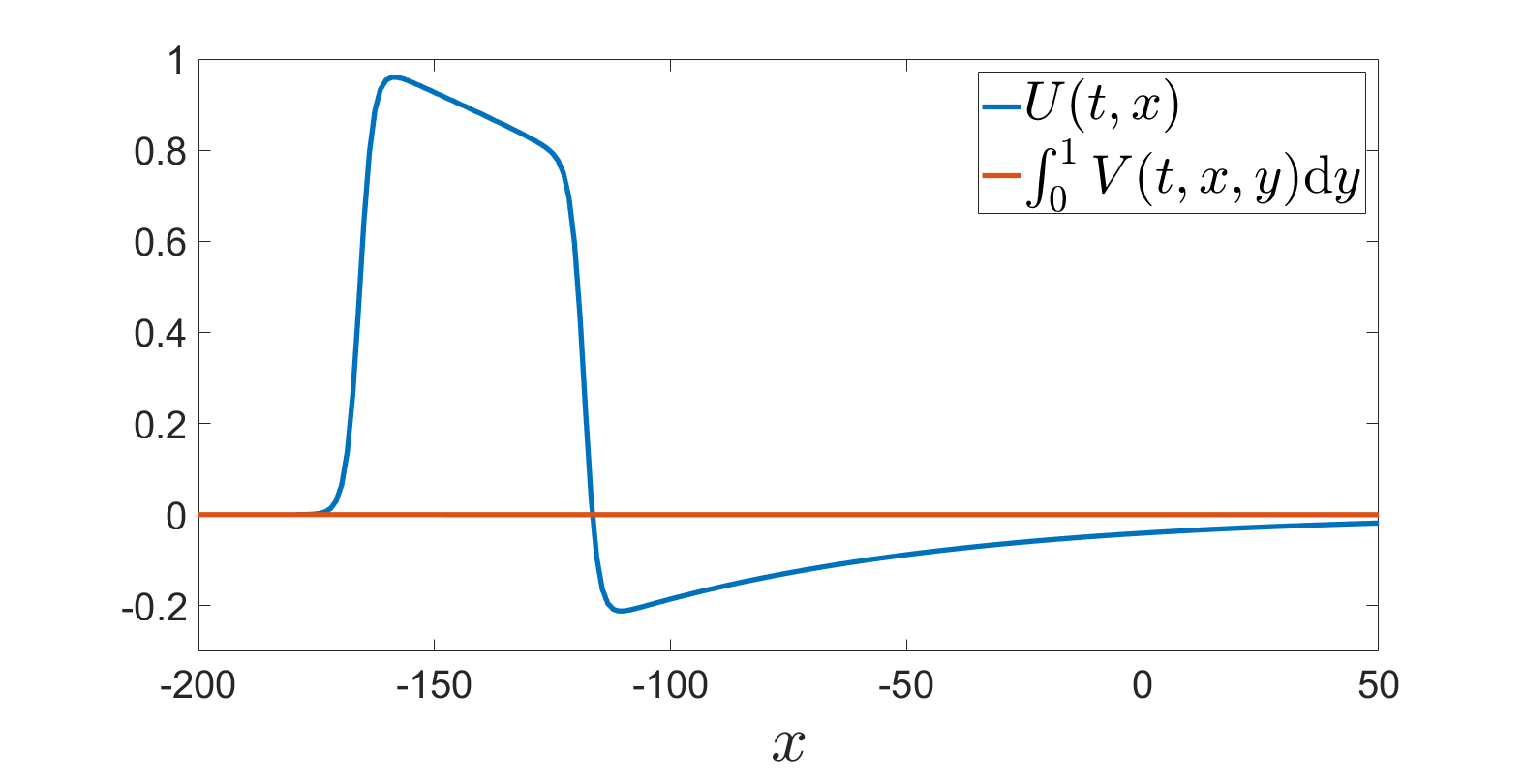}
\includegraphics[width=0.5\textwidth]{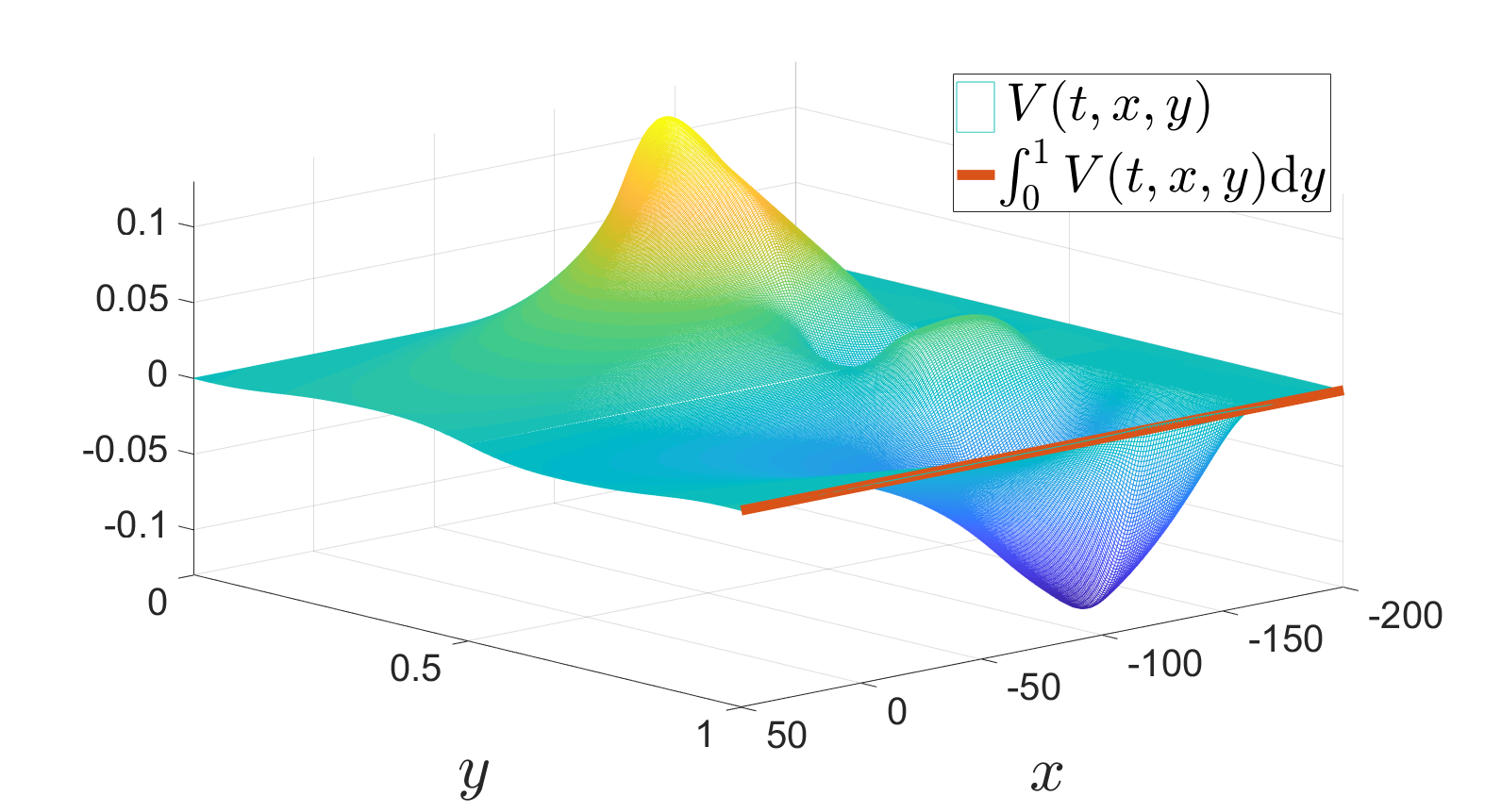}
\caption{Solution $(U,V)$ of the two-scale system \eqref{eq:system-lim} with $\alpha$ and $\beta$ as in \eqref{eq:param} and $b = d = \delta$. Left: the components $U$ and $\int_0^1 V(t,x,y) \dd y$. Right: the $V$-component in $xy$-plane (rotated by $180^\circ$) and its average.}
\label{fig:limit}
\end{figure}

Finally, we compare our results to the solution of the two-scale system \eqref{eq:system-lim}, see Figure \ref{fig:limit}.
We choose the initial conditions $U_0(x) = u(x)$ and $V_0(x,y) = \widetilde\alpha_1(y) v_1(x) + \widetilde\alpha_2(y) v_2(x) + \varphi(y) w(x)$. In order to plot the one-scale component $U(t,x)$ in one diagram with the two-scale component $V(t,x,y)$, see Figure \ref{fig:limit} (left), we average the solution $V$ over the periodicity cell $\rs$.
In our case
\begin{align*}
\int_0^1 \bv(x + ct ,y) \dd y = 0 \qquad \text{for all } x \in\rr, \, t \geq 0,
\end{align*}
since $\int_0^1 \beta (y) \dd y = 0$, cf.~ Remark \ref{rem:zero_spectrum}.\ref{rem:v_vanish}. In this sense we actually found an exemplary pulse solution with macroscopically vanishing inhibitor.

\subsection{Generalized pulse solution for the original system \eqref{eq:system-eps}}
\label{subsec:jumps}

In this example there is no inhibitor diffusion, $d(y) \equiv 0 $, and $b(y) \equiv b_0 > 0$ is constant such that Assumption \ref{assump:EllipticOrNot} is satisfied and the spectrum of $\cL = b_0 \, \mathrm{Id}$ consists of the only eigenvalue $b_0$. With this, any $\alpha \in \rmL^2(\rs)$ is an eigenfunction of $\cL$ and we choose
\begin{align}
\label{eq:param_2}
b_0 = 0.00001, \qquad
\alpha(y) = \left\lbrace
\begin{array}{ll}
+1 & \text{ if } y \in [0, 0.7) , \\
-1 & \text{ if } y \in [0.7 , 1) ,
\end{array}
\right. \qquad
\beta(y) = 0.003 \, \alpha(y) .
\end{align}
According to Remark \ref{rem:zero_spectrum}.\ref{rem:exact_sol}, the inhibitor $v^\eps$ of the generalized pulse solution $(u^\eps,v^\eps)$ of the original system
\begin{align}
\label{eq:num_eps_2}
u^\eps_t = u^\eps_{xx} + u^\eps (1 {-} u^\eps) (u^\eps {-} 0.15) - \alpha(\tfrac{x}{\eps}) v^\eps , \qquad
v^\eps_t = - b_0 v^\eps + \beta(\tfrac{x}{\eps}) u^\eps
\end{align}
exhibits oscillations, whereas the activator $u^\eps$ is independent of $\eps$, see Figure \ref{fig:eps_b0}.
\begin{figure}[h!]
\includegraphics[width=0.5\textwidth]{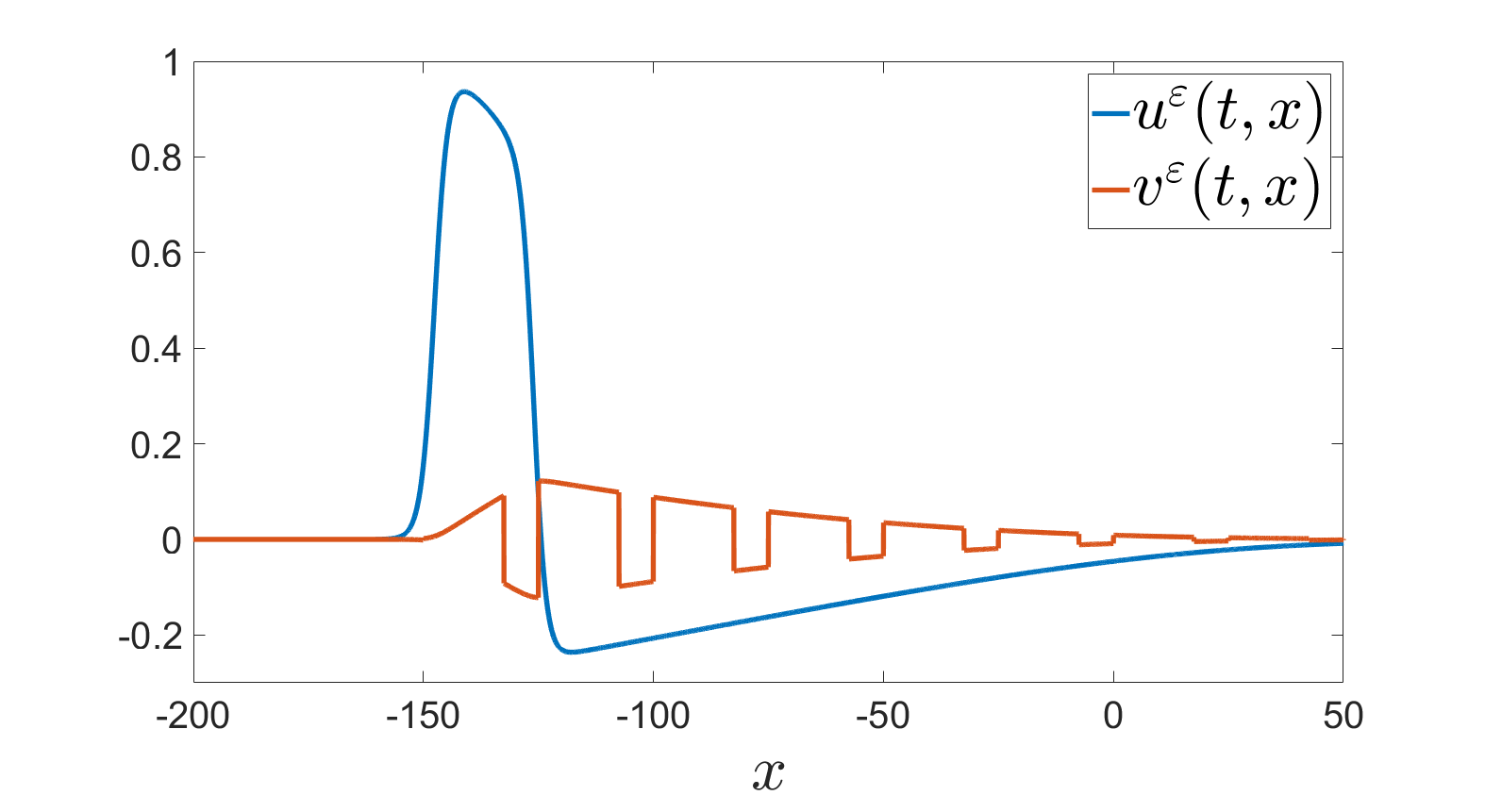}
\includegraphics[width=0.5\textwidth]{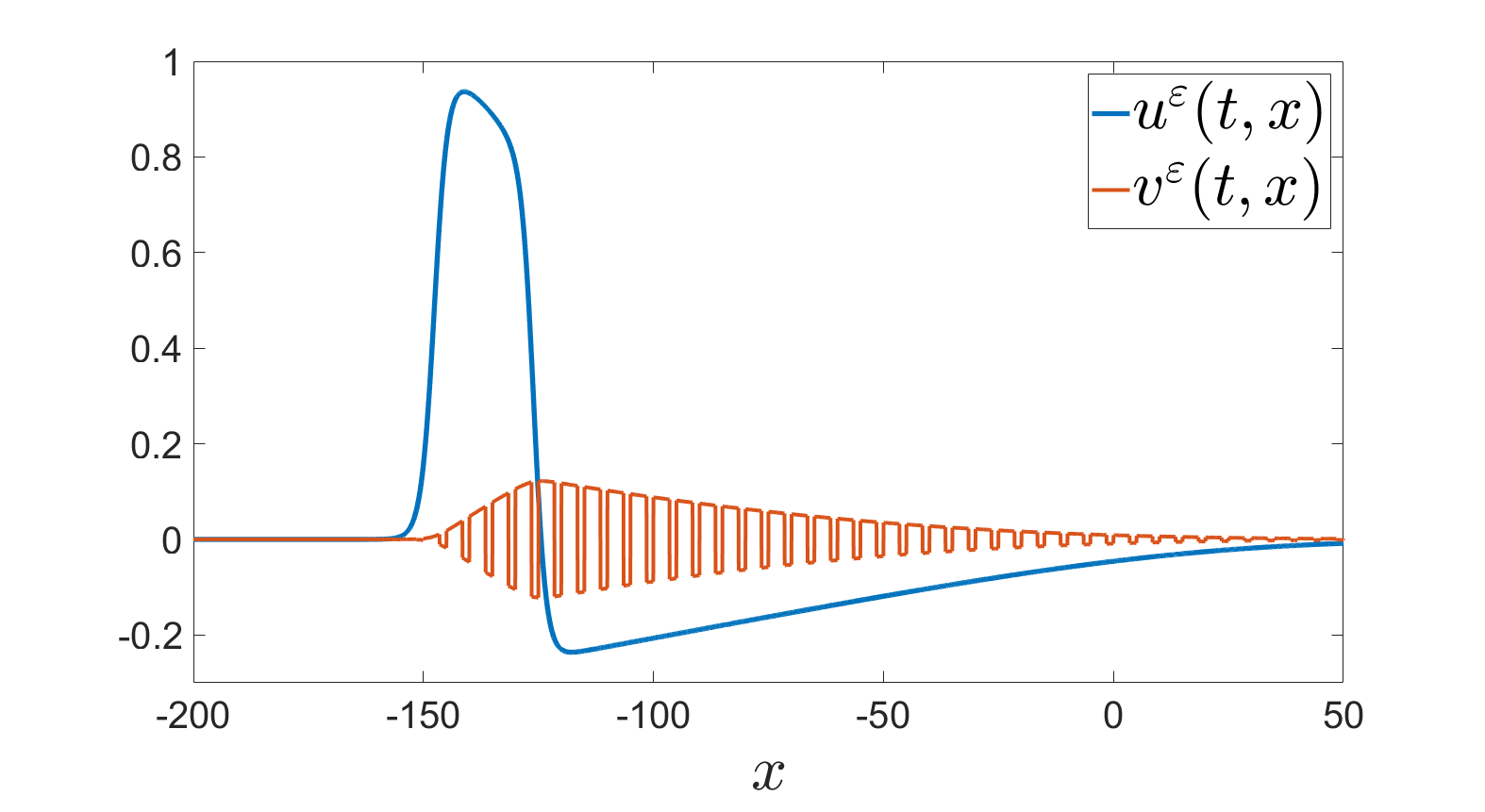}
\caption{Solution $(u^\eps, v^\eps)$ of the original system \eqref{eq:num_eps_2} with parameters \eqref{eq:param_2}. Left: $\eps = 25$. Right: $\eps = 5$.}
\label{fig:eps_b0}
\end{figure}
\begin{figure}[h!]
\includegraphics[width=0.5\textwidth]{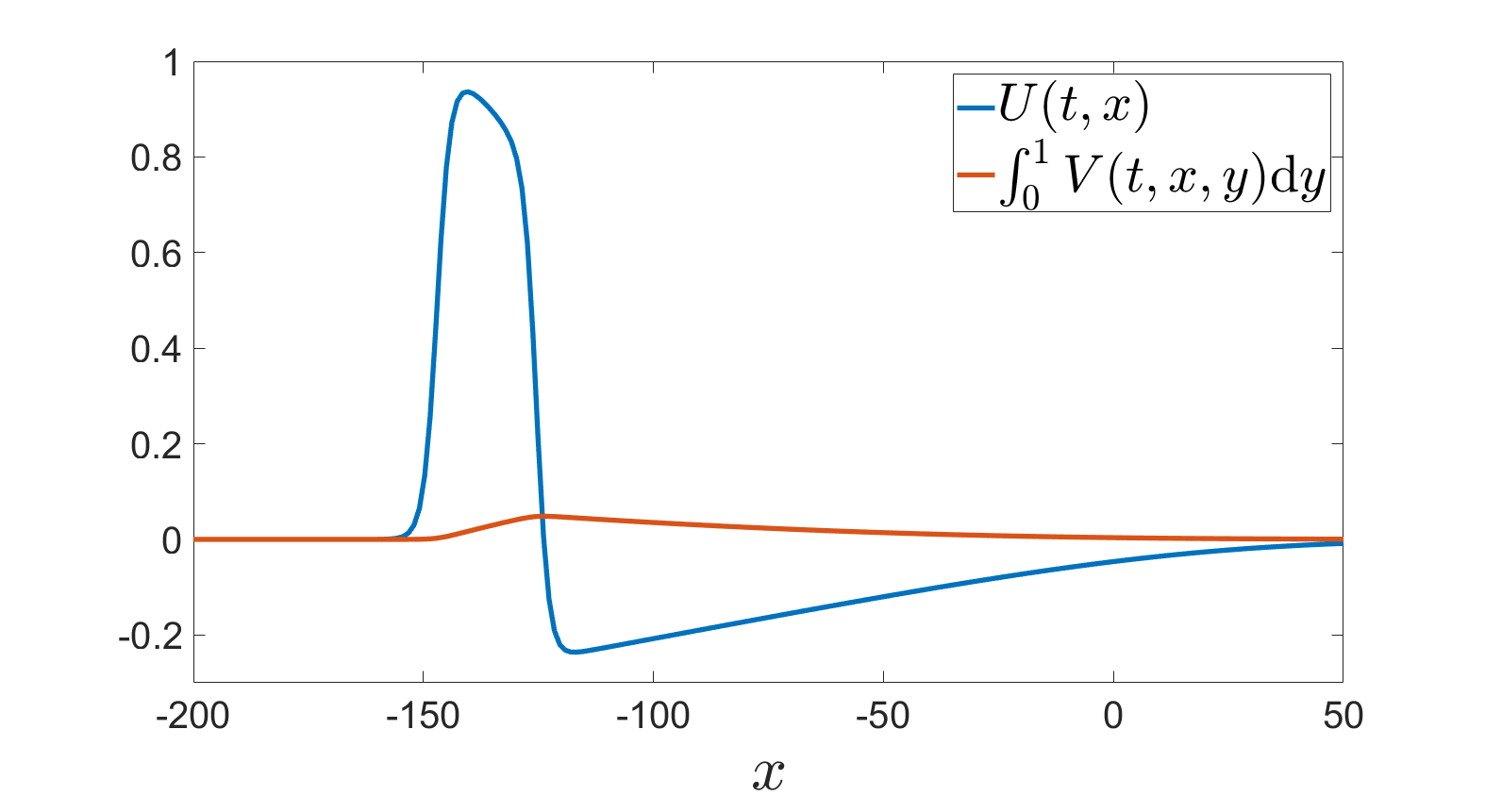}
\includegraphics[width=0.5\textwidth]{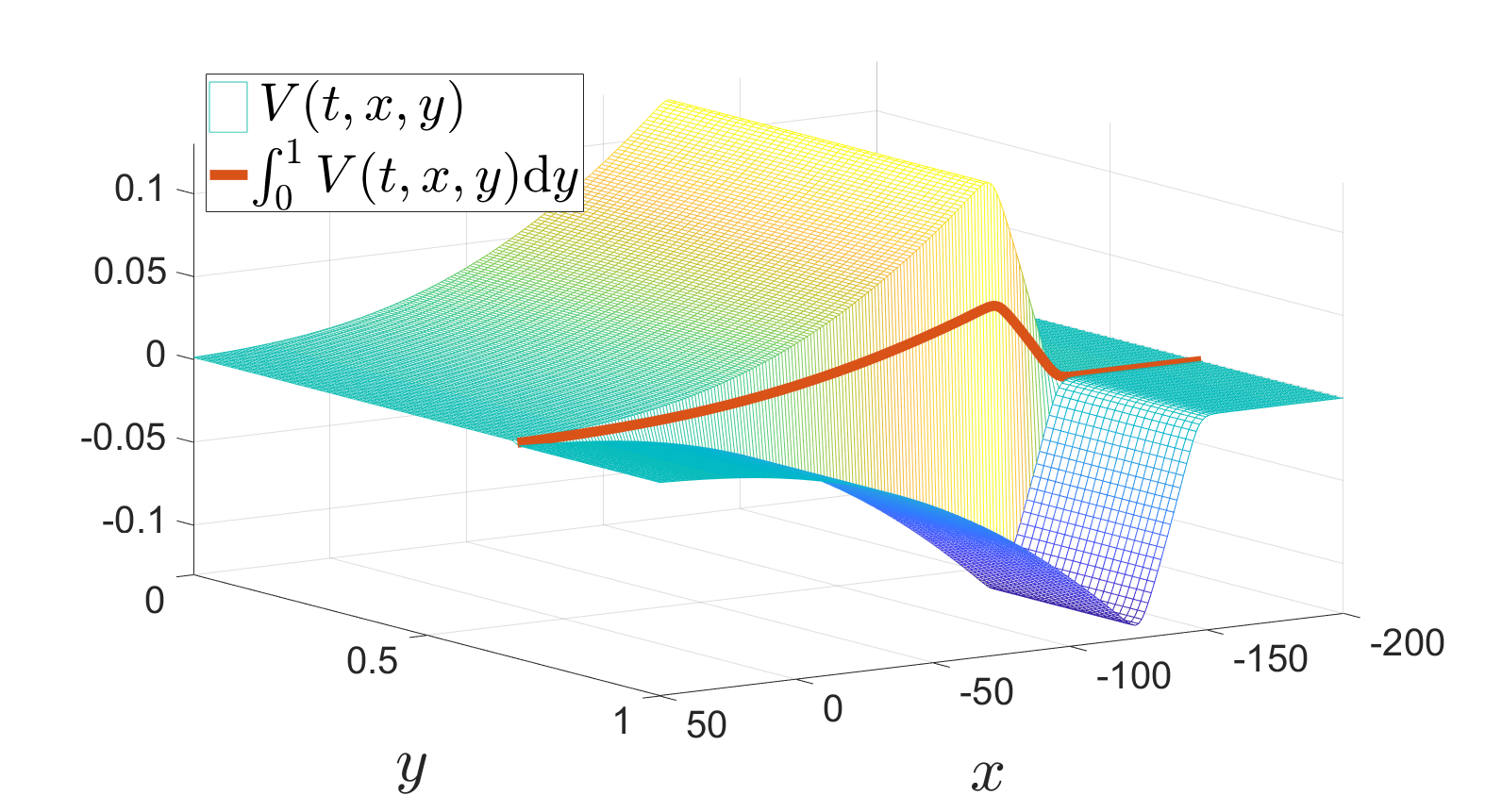}
\caption{Solution $(U, V)$ of the two-scale system \eqref{eq:system-lim} with parameters \eqref{eq:param_2} and $d=0$. Left: the components $U$ and $\int_0^1 V(t,x,y) \dd y$. Right: the $V$-component in $xy$-plane (rotated by $180^\circ$) and its average.}
\label{fig:limit_b0}
\end{figure}

Again, we observe a nice agreement with the two-scale pulse solution of the limit system \eqref{eq:system-lim}, see Figure \ref{fig:limit_b0}.
In this case the average of $V$ does not vanish, since $\int_0^1 \beta (y) \dd y \neq 0$. Due to \eqref{eq:genetal_pulse} and the relation $\int_0^1 \alpha(y) \dd y = 0.4$, we have
\begin{align*}
\int_0^1 V (t,x,y) \dd y = 0.4 \, v_1 (x + ct) \qquad \text{for all } x \in \rr, \, t \geq 0 ,
\end{align*}
where $v_1$ is given via the guiding system \eqref{eq:num_guide}.

\subsection{Continuous spectrum of $\cL$}
\label{subsec:contin_spec}

Let us consider the case where $\cL$ has only a continuous spectrum, which does not fit into the scope of our assumptions in Section \ref{sec:pulses}. In this case Theorem \ref{thm:error-est} still holds, but our method for the proof of two-scale pulses fails. However, we are able to present a numerical example which indicates that stable pulses also exist in this situation. Let us study the operator
$(\cL \varphi) (y) = b(y) \varphi$, where $b(y)$ is a positive and bounded non-constant function. The data are
\begin{align}
\label{eq:param_3}
b(y) = 0.001 (5 {+} 3\sin(2\pi y)), \qquad
\alpha(y) \equiv 1, \qquad
\beta(y) \equiv 0.003 .
\end{align}
We solve the original system for various $\eps$, see Figure \ref{fig:eps_b},
\begin{align}
\label{eq:num_eps_3}
u^\eps_t = u^\eps_{xx} + u^\eps (1 {-} u^\eps) (u^\eps {-} 0.15) -  v^\eps , \quad
v^\eps_t = - 0.001 \left( 5 {+} 3\sin(2\pi \tfrac{x}{\eps}) \right) v^\eps + 0.003 \!\cdot\! u^\eps .
\end{align}
\vspace*{-0.8em}
\begin{figure}[h!]
\includegraphics[width=0.5\textwidth]{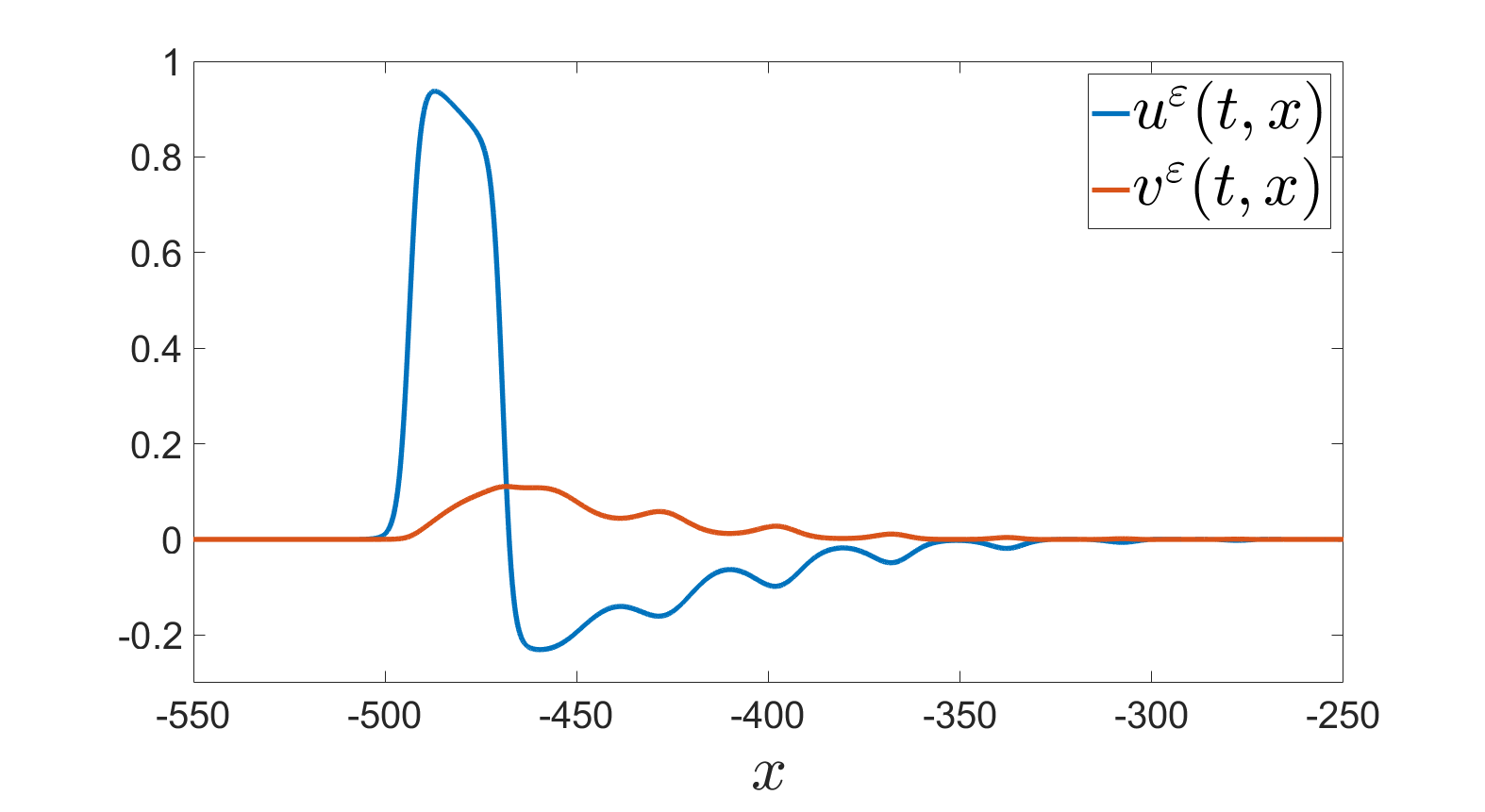}
\includegraphics[width=0.5\textwidth]{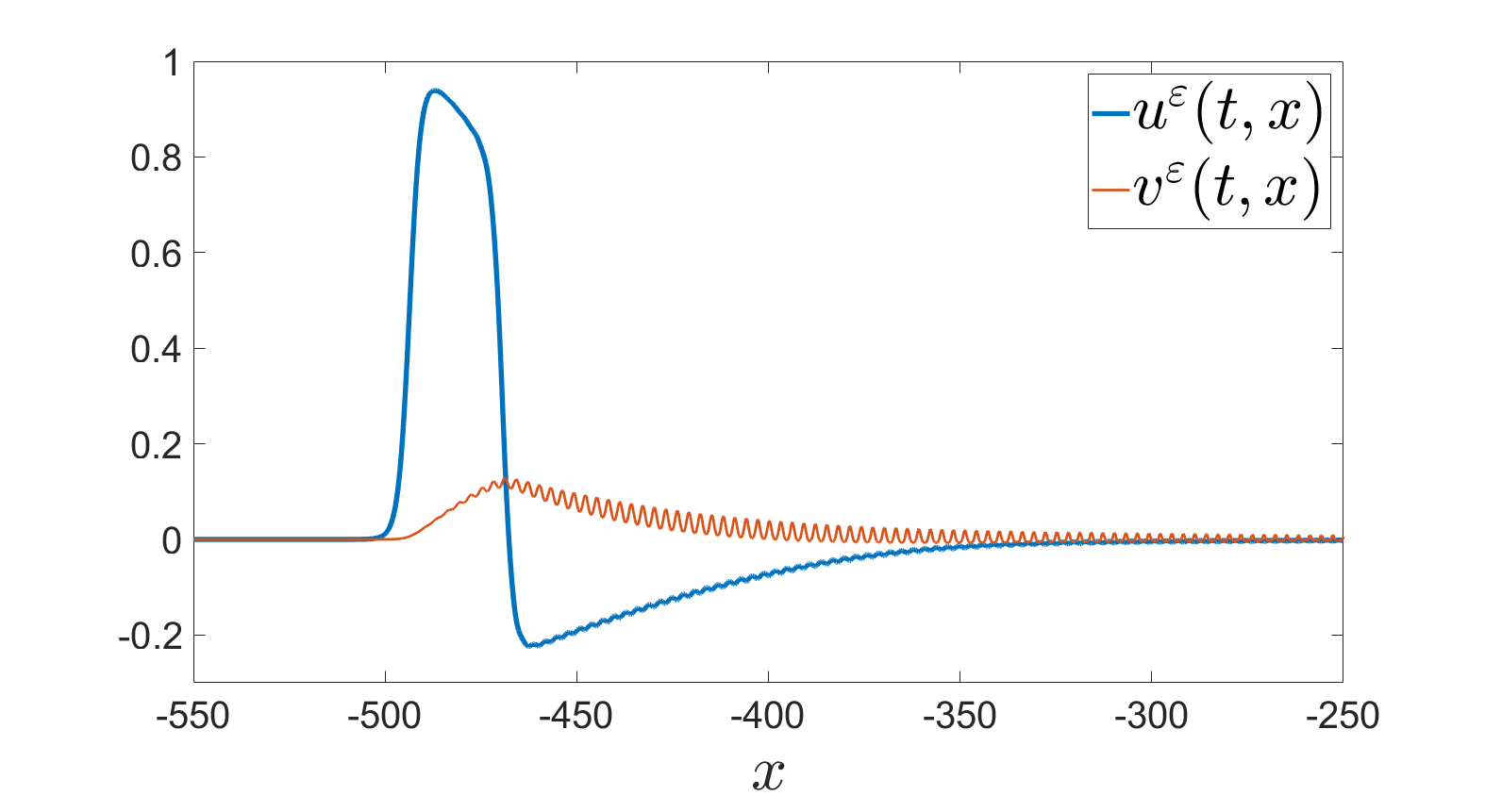}
\caption{Solution $(u^\eps, v^\eps)$ of the original system \eqref{eq:num_eps_3}. Left: $\eps = 30$. Right: $\eps = 3$.}
\label{fig:eps_b}
\end{figure}
\begin{figure}[h!]
\includegraphics[width=0.5\textwidth]{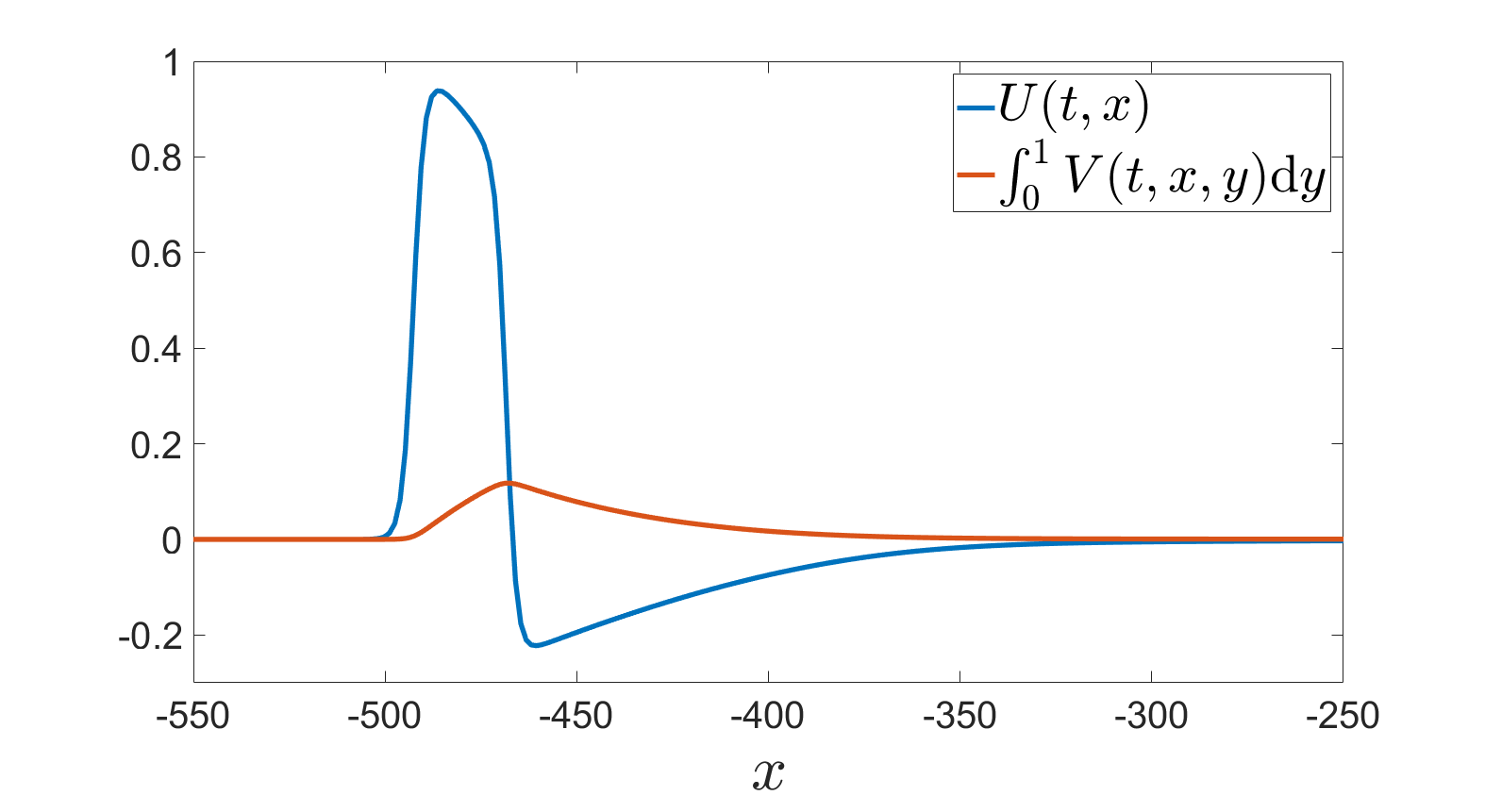}
\includegraphics[width=0.5\textwidth]{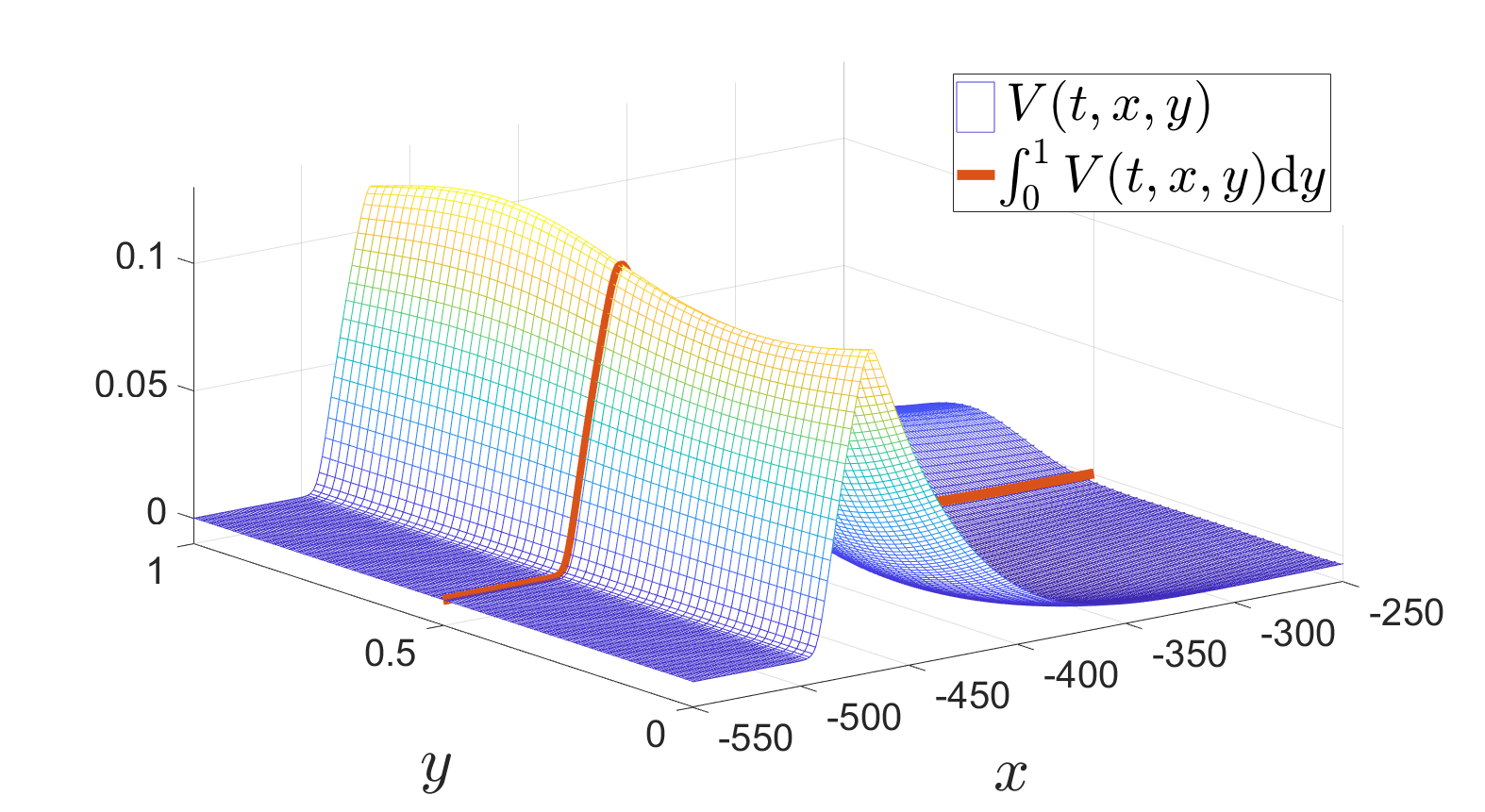}
\caption{Solution $(U,V)$ of the two-scale system \eqref{eq:system-lim} with parameters \eqref{eq:param_3}. Left: the components $U$ and $\int_0^1 V(t,x,y) \dd y$. Right: the $V$-component in $xy$-plane and its average.}
\label{fig:limit_b}
\end{figure}

The solution $(U,V)$ of the two-scale system \eqref{eq:system-lim} reproduces the effective behavior of the pulse $(u^\eps, v^\eps)$, see Figure \ref{fig:limit_b}.
In this case we do not have a suitable guiding system at hand, however, we choose as initial condition the pulse solution of the guiding system \eqref{eq:num_guide} with the parameters $\beta_1 = 0.003$ and $\lambda_1 = 0.005$.
Since the pulse has to evolve from the non-matching initial condition, we solve this example on the bigger interval $x \in [-700,700]$. The step sizes are $\dd x \approx 0.0427$ for the $\eps$-system and $\dd x \approx 1.3685$ for the limit system.

\appendix

\section{Auxiliary estimates}
\label{app:estimates}

The following lemma gives a standard proof for $\rmL^\infty$-boundedness for solutions of parabolic equations.
\begin{lemma}
\label{lemma:max-bound}
Let Assumptions~\ref{assump:coeff} and \ref{assump:initial-1} hold.
Any solution $(u^\eps,v^\eps)$ of \eqref{eq:system-eps} satisfies
\begin{align*}
\Vert u^\eps(t) \Vert_{\rmL^\infty(\rr)} + \Vert v^\eps(t) \Vert_{\rmL^\infty(\rr)} \leq C e^{\kappa t}
\qquad \text{for } t \geq 0 ,
\end{align*}
where the constants $C, \kappa \geq 0$ are independent of $\eps$ and $t$. Indeed, $C$ depends on $\Vert u^\eps_0 \Vert_{\rmL^\infty(\rr)}$ and $\Vert u^\eps_0 \Vert_{\rmL^\infty(\rr)}$, and $\kappa$ depends on $\max \lbrace \Vert \alpha \Vert_{\rmL^\infty(\rs)}, \Vert \beta \Vert_{\rmL^\infty(\rs)}, \Vert b \Vert_{\rmL^\infty(\rs)} \rbrace$ as well as the growth conditions of $f$ in Assumption \ref{assump:coeff}.\ref{assump:coeff2}.
\end{lemma}
\begin{proof}
For brevity we set $\alpha_\eps (x) : = \alpha (\tfrac{x}{\eps})$, etc., and define
\begin{align*}
M(t):= \max \lbrace 1, \Vert u^\eps_0 \Vert_{\rmL^\infty(\rr)}, \Vert v^\eps_0  \Vert_{\rmL^\infty(\rr)} \rbrace e^{2\kappa t} ,
\end{align*}
where $\kappa \in \rr$ is to be determined later.
We prove the lower bound $\min \lbrace u^\eps(t,x) ,  v^\eps(t,x) \rbrace \geq - M(t)$ and the upper bound $\max \lbrace u^\eps(t,x) ,  v^\eps(t,x) \rbrace \leq M(t)$ simultaneously.
First, we introduce the negative part for $\varphi \in C^0([0,T]; \rmL^2(\rr))$
\begin{align*}
(\varphi + M)_- (t,x) := \left\lbrace
\begin{array}{ll}
- (\varphi (t,x) + M(t)) & \text{ if } \varphi (t,x) \leq - M(t) \text{ for a.a.\ } x \in \rr , \\
0 & \text{ else}
\end{array} \right.
\end{align*}
and test the $u^\eps$- and $v^\eps$-equations in \eqref{eq:system-eps} with $- (u^\eps + M)_-$ and $-(v^\eps + M)_-$, respectively. Using $M_t = 2\kappa M$ and $M_x = 0$, integrating over $\rr$, and applying partial integration gives
\begin{align}
\label{eq:bound-1}
& \frac12 \frac{\dd}{\dd t} \left(
\Vert (u^\eps + M)_- \Vert^2_{\rmL^2(\rr)} +  \Vert (v^\eps + M)_- \Vert^2_{\rmL^2(\rr)} \right) \nonumber \\
& \leq \int_{\rr} \big\lbrace - \left( f(u^\eps) + \kappa M \right) (u^\eps + M)_-
- \left( -\alpha_\eps v^\eps + \kappa M \right) (u^\eps + M)_-  \nonumber \\
& \hspace{48pt}  - \left( - b_\eps v^\eps + \kappa M \right) (v^\eps + M)_-
- \left( \beta_\eps u^\eps + \kappa M \right) (v^\eps + M)_- \big\rbrace \dd x .
\end{align}
Secondly, we introduce the positive part
\begin{align*}
(\varphi - M)_+ (t,x) := \left\lbrace
\begin{array}{ll}
\varphi(t,x) - M(t) & \text{ if } \varphi (t,x) \geq M(t) \text{ for a.a.\ } x \in \rr , \\
0 & \text{ else}
\end{array} \right.
\end{align*}
and note that $(\varphi + M)_- \geq 0$ and $(\varphi - M)_+ \geq 0$ for all functions $\varphi \in C^0([0,T]; \rmL^2(\rr))$.
Testing \eqref{eq:system-eps} with $(u^\eps - M)_+$ and $(v^\eps - M)_+$ yields
\begin{align}
\label{eq:bound-2}
& \frac12 \frac{\dd}{\dd t} \left(
\Vert (u^\eps - M)_+ \Vert^2_{\rmL^2(\rr)} +  \Vert (v^\eps - M)_+ \Vert^2_{\rmL^2(\rr)} \right) \nonumber \\
& \leq \int_{\rr} \big\lbrace \left( f(u^\eps) - \kappa M \right) (u^\eps - M)_+
+ \left( -\alpha_\eps v^\eps - \kappa M \right) (u^\eps - M)_+ \nonumber \\
& \hspace{35pt} + \left( - b_\eps v^\eps - \kappa M \right) (v^\eps - M)_+
+ \left( \beta_\eps u^\eps - \kappa M \right) (v^\eps - M)_+  \big\rbrace \dd x .
\end{align}
Adding the estimates in \eqref{eq:bound-1} and \eqref{eq:bound-2} gives
\begin{align}
& \frac12 \frac{\dd}{\dd t} \left(
\Vert (u^\eps + M)_- \Vert^2_{\rmL^2(\rr)} +  \Vert (v^\eps + M)_- \Vert^2_{\rmL^2(\rr)} +
\Vert (u^\eps - M)_+ \Vert^2_{\rmL^2(\rr)} +  \Vert (v^\eps - M)_+ \Vert^2_{\rmL^2(\rr)} \right) \nonumber \\
& \leq \int_{\rr} \big\lbrace
- \left( f(u^\eps) + \kappa M \right) (u^\eps + M)_- +
\left( f(u^\eps) - \kappa M \right) (u^\eps - M)_+
\label{eq:bound-3a} \\
& \hspace{37pt} - \left( -\alpha_\eps v^\eps + \kappa M \right) (u^\eps + M)_-
+ \left( -\alpha_\eps v^\eps - \kappa M \right) (u^\eps - M)_+
\label{eq:bound-3b} \\
& \hspace{37pt} - \left( - b_\eps v^\eps + \kappa M \right) (v^\eps + M)_-
+ \left( - b_\eps v^\eps - \kappa M \right) (v^\eps - M)_+
\label{eq:bound-3c} \\
& \hspace{37pt} - \left( \beta_\eps u^\eps + \kappa M \right) (v^\eps + M)_-
+ \left( \beta_\eps u^\eps - \kappa M \right) (v^\eps - M)_+ \big\rbrace \dd x .
\label{eq:bound-3d}
\end{align}
The first term in \eqref{eq:bound-3a} is controlled as follows: $(u^\eps + M)_- > 0$ implies $u < 0$ and according to Assumption \ref{assump:coeff}.\ref{assump:coeff2} $f(u) \geq c_1 u - c_2$ with $c_1,c_2 \geq 0$. If $\kappa \geq \max \lbrace 2c_1, 2c_2 \rbrace$, then we have
\begin{align*}
- \left( f(u^\eps) + \kappa M \right) (u^\eps + M)_-
& \leq - (c_1 u + \tfrac{\kappa}{2} M) (u^\eps + M)_-
+ (c_2 - \tfrac{\kappa}{2}M)(u^\eps + M)_-  \\
& \leq \kappa |(u^\eps + M)_-|^2 .
\end{align*}
Analogously, the second term in \eqref{eq:bound-3a} is bounded by $\kappa |(u^\eps - M)_+|^2 $ for $\kappa \geq \max \lbrace 2c_3, 2c_4 \rbrace$. In the same manner we obtain that, if $\kappa \geq \Vert b \Vert_{\rmL^\infty(\rs)}$, then the sum of both terms in \eqref{eq:bound-3c} is bounded by $\kappa |(v^\eps + M)_-|^2 + \kappa |(v^\eps - M)_+|^2$. The mixed terms in \eqref{eq:bound-3b} can be controlled for $\kappa \geq \Vert \alpha \Vert_{\rmL^\infty(\rs)}$ via
\begin{align*}
& - \left( -\alpha_\eps v^\eps + \kappa M \right) (u^\eps + M)_-
+ \left( -\alpha_\eps v^\eps - \kappa M \right) (u^\eps - M)_+  \\
& \leq \kappa (|v^\eps| - M) \big( (u^\eps + M)_- + (u^\eps - M)_+ \big) \\
& \leq \left\lbrace
\begin{array}{ll}
0 & \text{ if } |v^\eps| < M ,\\
\kappa (v^\eps + M)_- \big( (u^\eps + M)_- + (u^\eps - M)_+ \big) & \text{ if } v^\eps \leq -M ,\\
\kappa (v^\eps - M)_+ \big( (u^\eps + M)_- + (u^\eps - M)_+ \big) & \text{ if } v^\eps \geq M \\
\end{array}
\right. \\
& \leq \kappa \left( |(v^\eps + M)_-|^2 + |(v^\eps - M)_+|^2 + |(u^\eps + M)_-|^2 + |(u^\eps - M)_+|^2 \right) .
\end{align*}
The mixed terms in \eqref{eq:bound-3d} are treated analogously.

Overall, choosing $\kappa = \max \lbrace 2c_1, 2c_2, 2c_3, 2c_4, \Vert \alpha \Vert_{\rmL^\infty(\rs)}, \Vert \beta \Vert_{\rmL^\infty(\rs)}, \Vert b \Vert_{\rmL^\infty(\rs)} \rbrace$ gives
\begin{align*}
& \frac12 \frac{\dd}{\dd t} \left(
\Vert (u^\eps + M)_- \Vert^2_{\rmL^2(\rr)} +  \Vert (v^\eps + M)_- \Vert^2_{\rmL^2(\rr)} +
\Vert (u^\eps - M)_+ \Vert^2_{\rmL^2(\rr)} +  \Vert (v^\eps - M)_+ \Vert^2_{\rmL^2(\rr)} \right) \\
& \leq 3 \kappa \left( \Vert (u^\eps + M)_- \Vert^2_{\rmL^2(\rr)} +  \Vert (v^\eps + M)_- \Vert^2_{\rmL^2(\rr)} +
\Vert (u^\eps - M)_+ \Vert^2_{\rmL^2(\rr)} +  \Vert (v^\eps - M)_+ \Vert^2_{\rmL^2(\rr)} \right) .
\end{align*}
By construction, the initial conditions satisfy $(u^\eps + M)_-(0,x) = (u^\eps - M)_+(0,x) = 0$ and $(v^\eps + M)_-(0,x) = (v^\eps - M)_+(0,x) = 0$ almost everywhere in $\rr$. Therefore, the application of Gr\"onwall's lemma implies $(u^\eps +M)_-(t,x) = (u^\eps - M)_+(t,x) = 0$ and $(v^\eps + M)_-(t,x) = (v^\eps - M)_+(t,x) = 0$ for all $t\geq 0$ and almost all $x\in\rr$. Hence, the desired $\rmL^\infty(\rr)$-bound holds uniformly with respect to $\eps$.
\end{proof}

\begin{rem}
\label{rem:A1}
With the same argumentation as in the proof of Lemma \ref{lemma:max-bound}, we obtain that any solution $(U,V)$ of \eqref{eq:system-lim} satisfies
\begin{align*}
\Vert U(t) \Vert_{\rmL^\infty(\rr)} + \Vert V(t) \Vert_{\rmL^\infty(\rr\times\rs)} \leq C e^{\kappa t}
\qquad \text{for } t \geq 0 ,
\end{align*}
where $C$ depends on $\Vert U_0 \Vert_{\rmL^\infty(\rr)}$ and $\Vert V_0 \Vert_{\rmL^\infty(\rr\times\rs)}$, and $\kappa$ is as in Lemma \ref{lemma:max-bound}.
\end{rem}

For completeness, we give the proof of the next lemma, which follows along the lines of \cite[Lem.~4.1]{Eck04}.
\begin{lemma}
\label{lemma:eck}
For every $g \in \rmH^1(\rr; \rmL^2(\rs))$, we set $\bar g(x) : = \int_0^1 g(x,y) \dd y $.
Then, the dual norm of $\RR g - \bar{g}$ is bounded via
\begin{align*}
\Vert \RR g - \bar g \Vert_{\rmH^1(\rr)^*} \leq \eps \Vert g \Vert_{\rmH^1(\rr; \rmL^2(\rs))} .
\end{align*}
\end{lemma}
\begin{proof}
We consider for arbitrary $\varphi \in \rmC^\infty_\mathrm{c} (\rr)$
\begin{align*}
\int_\rr (\RR g - \bar g) \varphi \dd x = \sum_{n \in \mathbb{Z}} \int_{\eps n}^{\eps(n+1)} (\RR g - \bar g) \varphi \dd x .
\end{align*}
Without loss of generality we set $n = 0$. Using the variable substitutions $x = \eps y$ and $x = \eps \tilde y$ gives
\begin{align*}
& \int_0^\eps g(x,\tfrac{x}{\eps}) \varphi(x) \dd x =  \eps \int_0^1 g(\eps y, y) \varphi(\eps y) \dd y
=  \eps \int_0^1 \int_0^1 g(\eps y, y) \varphi(\eps y) \dd y \dd \tilde y ,\\
& \int_0^\eps \bar g (x) \varphi(x) \dd x = \int_0^\eps \int_0^1  g (x,y) \varphi(x) \dd y \dd x
= \eps \int_0^1 \int_0^1  g (\eps \tilde y,y) \varphi(\eps \tilde y) \dd y \dd \tilde y .
\end{align*}
Subtracting both integrals and rearranging the integrands yields
\begin{align*}
\int_{0}^{\eps} (\RR g - \bar g) \varphi \dd x
& = \eps \int_{(0,1)^2} \left( g(\eps y , y) - g(\eps \tilde y , y)\right) \varphi(\eps y)
+ g(\eps \tilde y, y)\left( \varphi(\eps y) - \varphi(\eps \tilde y) \right)  \dd y\dd \tilde y .
\end{align*}
Exploiting the fundamental theorem of calculus
\begin{align*}
g(\eps y , y) - g(\eps \tilde y , y) = \eps \int_0^1 g_x(\eps  y t + (1 - t) \eps \tilde y , y) (y - \tilde y) \dd t
\end{align*}
as well as the variable transform
\begin{align*}
(t, \xi, \eta) = (t, t y + (1 - t) \tilde y , y - \tilde y)
\quad\text{with}\quad
\left\vert \mathrm{det} \left( \frac{\partial(t,\xi,\eta)}{\partial(t,y,\tilde y)} \right) \right\vert = 1 ,
\end{align*}
where $(t,\xi) \in (0,1)^2$ and $\eta \in (-1,1)$, yields with the Cauchy--Bunyakovsky--Schwarz inequality
\begin{align*}
\left\vert \int_{0}^{\eps} (\RR g - \bar g) \varphi \dd x  \right\vert
& \leq \eps^2 \left( \int_{(0,1)^2} \int_{-1}^1 |g_x(\eps \xi , \xi + (1 - t) \eta) \eta|^2 \dd t \dd \xi \dd \eta \right)^{\frac12}
\left( \int_0^1 |\varphi(\eps y)|^2 \dd y \right)^{\frac12} \\
& \quad + \eps^2 \left( \int_0^1 \int_{-1}^1 |\varphi_x(\eps \xi ) \eta|^2 \dd \xi \dd \eta \right)^{\frac12}
\left( \int_{(0,1)^2} |g(\eps \tilde y , y)|^2 \dd \tilde y \dd y \right)^{\frac12} \\
& \leq \eps^2 4 \Vert g \Vert_{\rmH^1((0,\eps); \rmL^2(\rs))} \Vert \varphi \Vert_{\rmH^1(0,\eps )}  .
\end{align*}
Summing up over all $n \in \mathbb{Z}$ and recalling the dense embedding of $\rmC^\infty_\mathrm{c} (\rr)$ into $\rmH^1(\rr)$ gives the desired estimate.
\end{proof}

\textbf{Acknowledgment.} The authors thank  Shalva Amiranashvili, Annegret Glitzky, Christian K\"uhn, and Alexander Mielke for helpful discussions and comments.
The research of S.R.\ was supported by \emph{Deutsche Forschungsgesellschaft} within SFB 910
\emph{Control of self-organizing nonlinear systems: Theoretical methods and concepts of application} via the project A5  \emph{Pattern formation in systems with multiple scales}.
The research of P.G.\ was supported by the DFG Heisenberg Programme, DFG project SFB 910,
and the Ministry of Education and Science of Russian Federation (agreement 02.a03.21.0008).

\footnotesize
\bibliographystyle{my_alpha}
\bibliography{Bib_Sina}
\normalsize

\end{document}